\documentclass[11pt, a4paper,oneside, reqno]{amsart}
\usepackage[english]{babel}
\usepackage{amsmath, amsthm, amsfonts, mathrsfs, amssymb, amscd}
\usepackage{bm}
\usepackage{mathtools}
\mathtoolsset{centercolon}
\usepackage{mathabx}
\usepackage{accents}

\usepackage[toc]{appendix}

\usepackage[shortlabels]{enumitem}
\usepackage[colorlinks, citecolor = blue, urlcolor={red}]{hyperref}
\usepackage{geometry}
\usepackage{bbm}

\allowdisplaybreaks

\newtheorem{theorem}{Theorem}

\newtheorem{lemma}[theorem]{Lemma}

\theoremstyle{remark} 
\newtheorem{remark}[theorem]{Remark}

\theoremstyle{definition} 

\newtheorem{assumption}[theorem]{Assumption}

\numberwithin{theorem}{section}
\numberwithin{equation}{section}

\def\R{{\mathbb R}}

\renewcommand{\P}{{\mathbb P}}

\newcommand{\om}{\omega}
\newcommand{\Om}{\Omega}

\newcommand{\abs}[1]{\vert #1 \vert}

\newcommand{\norm}[1]{\Vert #1 \Vert}
\newcommand{\inner}[2]{\langle #1,\, #2 \rangle}

\newcommand{\BS}{{\rm BS}}

\newcommand{\T}{\mathbb{T}}

\newcommand{\mc}{\mathcal}

\newcommand{\Ls}{\mathbb{L}}
\newcommand{\Hs}{\mathbb{H}}
\newcommand{\Hinf}{H^{\infty}}
\newcommand{\RR}{\mathbb{R}}

\newcommand{\NN}{\mathbb{N}}

\newcommand{\BB}{\mathbb{B}}

\newcommand{\half}{\frac{1}{2}}

\newcommand{\RRd}{\RR^d}

\renewcommand{\d}{\partial}
\newcommand{\del}{\Delta}
\newcommand{\grad}{\nabla}

\newcommand{\eps}{\varepsilon}

\newcommand{\gam}{\gamma}

\newcommand{\loc}{{\rm loc}}
\renewcommand{\tilde}[1]{\widetilde{#1}}
\renewcommand{\hat}[1]{\widehat{#1}}

\newcommand{\dd}{\hspace{2pt}\mathrm{d}}

\DeclareMathOperator{\ind}{\mathbf{1}}

\DeclareMathOperator{\SMR}{SMR}

\begin{document}

\author[Goodair]{Daniel Goodair}
\address{Daniel Goodair \hfill\break\indent École Polytechnique Fédérale de Lausanne \hfill\break\indent Bâtiment MA \hfill\break\indent
Station 8\hfill\break\indent
1015 Lausanne, Switzerland }
\email{daniel.goodair@epfl.ch}

\author[Roodenburg]{Floris Roodenburg}
\address{Floris Roodenburg \hfill\break\indent
Delft Institute of Applied Mathematics \hfill\break\indent
Delft University of Technology \hfill\break\indent
P.O. Box 5031 \hfill\break\indent
2600 GA Delft, The Netherlands}
\email{f.b.roodenburg@tudelft.nl}

\title{GLOBAL WELL-POSEDNESS FOR THE 2D STOCHASTIC
HYPOVISCOUS NAVIER-STOKES EQUATIONS}


\subjclass[2020]{35Q35, 60H15, 76D03, 76M35}
\keywords{2D stochastic hypoviscous Navier--Stokes equations, stochastic maximal regularity}

\thanks{The second author is supported by the VICI grant VI.C.212.027 of the Dutch Research Council (NWO)}

\begin{abstract}
We study stochastic hypoviscous Navier--Stokes equations on the torus with dissipation $(-\Delta)^\gamma$ for $\gam\in (\frac 12,1]$ and multiplicative noise. Relying on stochastic maximal regularity results for the linear equation, we establish local well-posedness for this problem in arbitrary dimensions with initial data in a range of scaling-critical Besov spaces. With the aid of $L^q$-energy estimates for the vorticity equation, we also prove global well-posedness of the stochastic hypoviscous Navier--Stokes equation in 2D with linear multiplicative noise.
\end{abstract}

\maketitle
\setcounter{tocdepth}{1}
\tableofcontents

\section{Introduction}
We consider the following incompressible stochastic fractional Navier--Stokes equations on the $d$-dimensional torus $\T^d$ for $d\geq 2$:
\begin{equation}\label{eq:FSNSE_intro}
\left\{
\begin{aligned}
    &\dd u + \big[(-\del)^\gam u+(u\cdot\grad) u+ \grad \pi\big]\dd t  =\sum_{n\geq 1}\big[g_n(\cdot, u)-\grad \pi_n\big]\dd W^n_t&\text{ on }& \T^d,\\
    &\grad \cdot u = 0&\text{ on }& \T^d,\\
    &u|_{t=0} = u_0,
\end{aligned}
    \right.
\end{equation}
where $u =(u^k)_{k=1}^d:[0,\infty)\times \Omega\times \T^d\to \R^d$ and $\pi, \pi_n: [0,\infty)\times \Omega \times \T^d\to \R$ denote the unknown velocity field and pressures, respectively. In addition, $g=(g_n)_{n\geq 1}$ and $u_0$ are the stochastic forcing term and initial value, respectively, for which we specify the precise conditions in Section \ref{sec:mainresuls}. Furthermore,  $(W^n_t: t\geq 0)_{n\geq 1}$ is a given sequence of standard independent Brownian motions and
\begin{equation}\label{eq:def_convection}
    (u\cdot \grad) u =\big(\sum_{j=1}^d u^j \d_j u^k\big)_{k=1}^d
\end{equation}
denotes the convection term.

Throughout, we assume $\gam\in (\frac 12, 1]$.  The case $\gam=1$ is the classical stochastic Navier--Stokes equation, whereas
$\frac12<\gam<1$ and $\gam>1$ are referred to as the hypoviscous and hyperviscous regimes, respectively. The hypoviscous case has weaker dissipation than the classical equation, while in the hyperviscous case the additional smoothing helps in the analysis of the equation. We will only deal here with the more difficult hypoviscous regime via a functional-analytic framework based on critical spaces and stochastic maximal regularity for the differential operator $(-\del)^\gam$.

For deterministic parabolic equations the theory of critical spaces has been developed systematically in \cite{PSW18} and is applied to the deterministic Navier--Stokes equations in \cite{PW18}. For stochastic equations, the corresponding critical-space theory based on stochastic maximal regularity was developed in \cite{AV22,AV22_part2}, see also the survey \cite{AV25_survey}. These ideas were
used in \cite{AV24_SNSE} to study stochastic Navier--Stokes equations on the torus. In the present paper, we use stochastic maximal regularity for the nonlocal operator $(-\Delta)^\gam$ and prove local well-posedness for \eqref{eq:FSNSE_intro} on $\T^d$ with arbitrary $\gam\in (\frac 12, 1]$ and initial data in critical Besov spaces, see Section \ref{subsec:scaling} for a scaling analysis and Theorem \ref{thm:localWP} for our local well-posedness result. Based on $L^q$-energy estimates for the vorticity equation, we also prove global well-posedness of \eqref{eq:FSNSE_intro} on $\T^2$ with arbitrary $\gam\in(\frac 12, 1]$, see Theorem \ref{thm:globalWPvelocity}.
We refer to \cite{De16, Wu04,YWS22} for several related works about the (stochastic) fractional Navier--Stokes equations. 

\subsection{Scaling and criticality}\label{subsec:scaling}
Before stating our main results in the next section, we carry out a scaling argument for the stochastic fractional Navier--Stokes equations. The scaling analysis provides insight into the minimal smoothness assumptions for local well-posedness in Theorem \ref{thm:localWP}. 

On $\RR^d$ it is straightforward to check that (smooth) solutions to \eqref{eq:FSNSE_intro} with $g\equiv 0$ (and ignoring the pressure) are invariant under the scaling
\begin{equation*}
    u(t,x) \mapsto u_\lambda(t,x):= \lambda^{1-\frac{1}{2\gam}}u(\lambda t, \lambda^{\frac{1}{2\gam}}x), \quad (t,x)\in \RR_+\times \RR^d\, \text{ and }\, \lambda>0.
\end{equation*}
Note that for $\gam=1$ this coincides with the scaling for the Navier--Stokes equations, see for instance \cite{AV24_SNSE, PW18}. In the literature (see, e.g., \cite{AV25_survey, LR16, PW18, Tr13}), a Banach space is called \emph{critical} if functions in this Banach space are invariant under the following map for the initial data:
\begin{equation*}
   u_0\mapsto u_{0,\lambda} := \lambda^{1-\frac{1}{2\gam}}u_0(\lambda^{\frac{1}{2\gam}}\cdot).
\end{equation*}
The critical spaces within the scale of Lebesgue and (homogeneous) Besov spaces are
\begin{equation*}
    \dot{B}_{q,p}^{\frac{d}{q}+1-2\gam}(\RRd;\RRd)\quad \text{ and }\quad L^{\frac{d}{2\gam-1}}(\RRd;\RRd).
\end{equation*}
Indeed, these spaces satisfy the scaling
\begin{align*}
    \|u_{0,\lambda}\|_{\dot{B}_{q,p}^{\frac{d}{q}+1-2\gam}(\RRd;\RRd)}&\eqsim\lambda^{1-\frac{1}{2\gam}}(\lambda^{\frac{1}{2\gam}})^{(\frac{d}{q}+1-2\gam)-\frac{d}{q}}\|u_0\|_{\dot{B}_{q,p}^{\frac{d}{q}+1-2\gam}(\RRd;\RRd)} = \|u_0\|_{\dot{B}_{q,p}^{\frac{d}{q}+1-2\gam}(\RRd;\RRd)},\\
     \|u_{0,\lambda}\|_{L^{\frac{d}{2\gam-1}}(\RRd;\RRd)}&\eqsim \lambda^{1-\frac{1}{2\gam}} (\lambda^{\frac{1}{2\gam}})^{-d (\frac{d}{2\gam-1})^{-1}}\|u_0\|_{L^{\frac{d}{2\gam-1}}(\RRd;\RRd)} = \|u_0\|_{L^{\frac{d}{2\gam-1}}(\RRd;\RRd)},
\end{align*}
where the implicit constant is independent of $\lambda>0$. Again, we note that for $\gam=1$ these spaces coincide with the well-known critical spaces for the Navier--Stokes equations: $\dot{B}_{q,p}^{d/q-1}(\RRd;\RRd)$ and $L^d(\RRd;\RRd)$.

In the case of fractional Navier--Stokes on the torus, we can replace $\RRd$ by $\T^d$ and then the argument above works locally in space. Therefore, we can still determine the critical space on the torus. The critical spaces determined via scaling coincide with those obtained via the abstract setting in \cite{PSW18, PW18} for the deterministic case. The stochastic variant of the theory of critical spaces was developed in \cite{AV22, AV22_part2} and we will use their set-up in this paper to study well-posedness of \eqref{eq:FSNSE_intro}.

\section{Statement of the main results}\label{sec:mainresuls}
In this section, we first introduce the necessary function spaces in Section \ref{subsec:Helmholtz}. Afterwards, in Sections \ref{subsec:LWP} and \ref{subsec:GWP} we state our main results about local and global well-posedness of the stochastic hypoviscous Navier--Stokes equations.

\subsection{The Helmholtz projection and function spaces}\label{subsec:Helmholtz} 
We introduce the Helmholtz projection and the corresponding divergence-free function spaces. For more details, we refer to, e.g., \cite[Chapter III]{Ga11} or \cite[Chapter 2]{robinson2016three}.\\

For any $s\in\mathbb{R}$, $p\in[1,\infty]$, $q\in(1,\infty)$ and integers
$d,m\geq 1$, let
\[
L^q(\mathbb{T}^d;\mathbb{R}^m),\quad
H^{s,q}(\mathbb{T}^d;\mathbb{R}^m)
\quad\text{and}\quad
B^s_{q,p}(\mathbb{T}^d;\mathbb{R}^m)
\]
denote the vector-valued \emph{Lebesgue}, \emph{Bessel-potential} and \emph{Besov spaces} on $\mathbb{T}^d$
with values in $\mathbb{R}^m$, respectively. We refer to \cite[Chapter 3]{ST87} for more details on periodic function spaces.

Let $f=(f^n)_{n=1}^d\in C^\infty(\mathbb{T}^d;\mathbb{R}^d)$ and let  $\widehat{f^n}(k)$ be the $k$-th Fourier coefficient. We denote by $\mathbb{P}$ the \emph{Helmholtz projection}, which is given by
$\mathbb{P}f=((\mathbb{P}f)^n)_{n=1}^d$, where
\begin{equation*}
\widehat{(\mathbb{P}f)^n}(k)
:=
\widehat{f^n}(k)
-
\sum_{j=1}^d
\frac{k_jk_n}{|k|^2}\widehat{f^j}(k),
\qquad
k\in\mathbb{Z}^d\setminus\{0\},
\qquad
\widehat{(\mathbb{P}f)^n}(0)
:=
\widehat{f^n}(0).
\end{equation*}
In particular, it holds that $\operatorname{div}(\mathbb{P}f)=0$ for all $f\in C^\infty(\mathbb{T}^d;\mathbb{R}^d)$.
By duality the Helmholtz projection extends to a mapping 
$\mathbb{P}:
\mathscr{D}'(\mathbb{T}^d;\mathbb{R}^d)
\to
\mathscr{D}'(\mathbb{T}^d;\mathbb{R}^d)$, where $\mathscr{D}'(\mathbb{T}^d;\mathbb{R}^d)
=
C^\infty(\mathbb{T}^d;\mathbb{R}^d)'$.

Let $s\in\mathbb{R}$ and  $q\in(1,\infty)$. Then by standard Fourier multiplier theorems
(see \cite[Theorem~5.7.11]{HNVW16}), $\mathbb{P}$ restricts uniquely to a bounded linear operator
\begin{equation*}
\mathbb{P}:
H^{s,q}(\mathbb{T}^d;\mathbb{R}^d)
\to
H^{s,q}(\mathbb{T}^d;\mathbb{R}^d)
.
\end{equation*}
For $s$ and $q$ as specified above and
$Y\in\{L^q,H^{s,q},B^s_{q,p}\}$, we define the divergence-free spaces by
\begin{equation*}
\mathbb{Y}(\mathbb{T}^d;\RRd)
:=
\left\{
f\in Y(\mathbb{T}^d;\mathbb{R}^d)
:
\operatorname{div}f=0
\text{ in }\mathscr{D}'(\mathbb{T}^d)
\right\},
\qquad
\|f\|_{\mathbb{Y}(\mathbb{T}^d;\RRd)}
:=
\|f\|_{Y(\mathbb{T}^d;\mathbb{R}^d)}.
\end{equation*}
Note that $\mathbb{Y}(\mathbb{T}^d;\RRd)
=
\mathbb{P}(Y(\mathbb{T}^d;\mathbb{R}^d))
$
for $Y\in\{L^q,H^{s,q},B^s_{q,p}\}$ and $s,q$ as above.
By standard interpolation theory
(see \cite[Section~3.6]{ST87} and \cite[Theorem~1.2.4]{Tr78}),
\begin{equation*}
\begin{aligned}
\Hs^{s,q}(\mathbb{T}^d;\RRd)
&=
\bigl[
\Hs^{s_0,q}(\mathbb{T}^d;\RRd),
\Hs^{s_1,q}(\mathbb{T}^d;\RRd)
\bigr]_\theta,
\\
\mathbb{B}^s_{q,p}(\mathbb{T}^d;\RRd)
&=
\bigl(
\Hs^{s_0,q}(\mathbb{T}^d;\RRd),
\Hs^{s_1,q}(\mathbb{T}^d;\RRd)
\bigr)_{\theta,p},
\end{aligned}
\end{equation*}
for all $s_0,s_1\in\mathbb{R}$ with $s_0\neq s_1$, $\theta\in(0,1)$, $
s:=(1-\theta)s_0+\theta s_1$, and $p,q\in(1,\infty)$.

\subsection{Local well-posedness results}\label{subsec:LWP} As is standard in the literature, we will apply the Helmholtz projection to study well-posedness for the Navier--Stokes equations. For more details on recovering the pressure from a solution of the projected equation, see for example \cite[Chapter 5]{robinson2016three} and \cite[Proposition 3.3]{crisan2022solution}. In general, the pressure term will not be of finite variation but rather a more general semimartingale, as explained in \cite[Section 3c]{street2021semi}. Hence, the inclusion of a stochastically integrated pressure term in \eqref{eq:FSNSE_intro}. After applying the Helmholtz projection $\P$ to \eqref{eq:FSNSE_intro} we obtain the  problem
\begin{equation}\label{eq:abstract_eq}
       \left\{
        \begin{aligned}
            &\dd u + A u \dd t = F(u)\dd t + G(u)\dd W_{\ell^2}&\text{ for }&t\geq 0,\\
            &u|_{t=0}=u_0,
        \end{aligned}
        \right.
\end{equation}
with 
\begin{equation*}
    \begin{aligned}
        Au&:= (-\del)^\gam u,\quad F(u):=-\P [\grad \cdot (u\otimes u)]\quad \text{ and }\quad G(u):= (\P[g_n(\cdot, u)])_{n\geq1},
    \end{aligned}
\end{equation*}
where we have used $\P (-\del)^\gam u = (-\del)^\gam \P u = (-\del)^\gam u$ on $\T^d$ and $\grad \cdot (u\otimes u) = (u\cdot \grad)u$ (conservative form for the Navier--Stokes nonlinearity) since $\grad\cdot u =0$. We recall that $W_{\ell^2}$ is a cylindrical Brownian motion in the Hilbert space $\ell^2$, see \cite[Section 2.6]{AV25_survey} for details.\\

We introduce some notation and assumptions. Let $(\Om, \mc{F}, (\mc{F}_t)_{t\geq 0}, \mathbf{P})$ be a fixed  complete filtered probability space and let $\mc{P}$ and $\mc{B}$ denote the progressive and Borel $\sigma$-algebras, respectively. Moreover, we will always consider dimensions $d\geq 2$.

\begin{assumption}\label{ass:X0X1_g}
 Let $\gam\in (\frac 12,1]$, $s\in[\gam, 2\gam)$ and $q\in (2,\infty)$. We set
\begin{equation*}
    X_0:= \Hs^{-s,q}(\T^d;\RRd)\quad \text{ and }\quad X_1:=\Hs^{-s+2\gam,q}(\T^d;\RRd)
\end{equation*}
and consequently
\begin{equation*}
    X_{\beta}: =[X_0, X_1]_\beta = \Hs^{-s+2\gam\beta,q}(\T^d;\RRd),\qquad \beta\in(0,1).
\end{equation*}
Furthermore, we assume that $g:=(g_n)_{n\geq1}:\T^d\times \RR^d\to \ell^2(\NN_1; \RRd)$ is a measurable mapping which satisfies $g(\cdot, 0)\in L^\infty(\T^d;\ell^2(\NN_1;\RRd))$ and is globally Lipschitz uniformly in $x$, i.e.,
\begin{equation*}
    \sup_{x\in \T^d}\|g(x,u)-g(x,v)\|_{\ell^2(\NN_1;\RRd)}\lesssim |u-v|, \qquad u,v\in \RRd.
\end{equation*}
\end{assumption}

We obtain the following local well-posedness result for the stochastic hypoviscous Navier--Stokes equations.

\begin{theorem}[Local well-posedness]\label{thm:localWP}
    Let $\gam \in (\frac 12, 1]$, $s\in [\gam, 2\gam)$ and suppose that Assumption \ref{ass:X0X1_g} holds. Furthermore, assume that $p,q\in (2,\infty)$ satisfy
        \begin{equation*}
        \max\{|s-1|, 3\gam-s-1\}<\frac dq < 4\gam-s-1\quad \text{ and }\quad \frac{2}{p}+\frac{d}{q\gam}\leq 4-\frac{s+1}{\gam}.
    \end{equation*}
    Moreover, set $\kappa:=-1 + \frac{p}{2}\big(4-\gam^{-1}(\frac{d}{q}+s+1)\big)$. Then for all $u_0 \in L^0_{\mc{F}_0}(\Om; \BB_{q,p}^{\frac{d}{q}+1-2\gam}(\T^d;\RRd))$ there exists a unique maximal solution $(u,\sigma)$ to \eqref{eq:abstract_eq} satisfying $\sigma>0$ a.s. and
    \begin{align*}
        u&\in H^{\theta,p}_{\loc}([0,\sigma),w_{\kappa}; \Hs^{-s+2\gam(1-\theta),q}(\T^d; \RRd)) \,\,\text{a.s. for all }\theta\in [0,\tfrac 12),\\
        u&\in C([0, \sigma); \BB_{q,p}^{\frac{d}{q}+1-2\gam}(\T^d; \RRd))\,\, \text{ a.s.}
    \end{align*}
    The solution $(u,\sigma)$ instantaneously regularises in time and space:
\begin{align*}
u&\in L^r_\loc((0,\sigma);\Hs^{\gam,\zeta}(\T^d;\RRd))
\quad\text{a.s.}, &&r,\zeta\in(2,\infty),\\
u&\in C^{\theta-\eps}_\loc((0,\sigma);
\Hs^{-\gam+2\gam(1-\theta),\zeta}(\T^d;\RRd))
\quad\text{a.s.}, &&
\theta\in(0,\tfrac 12),\ \eps\in(0,\theta),\ \zeta\in(2,\infty).
\end{align*}
\end{theorem}

The proof of Theorem \ref{thm:localWP} is given in Section  \ref{sec:LWP} and relies on stochastic maximal regularity for the linear equation (obtained via the bounded $\Hinf$-calculus for the leading differential operator $A$) and appropriate estimates on the nonlinearities $F$ and $G$.

\begin{remark}
    We make the following remarks about Theorem \ref{thm:localWP}.
    \begin{enumerate}[(i)]
        \item The regularity $\frac{d}{q}+1-2\gam$ for the initial data is
precisely the smoothness exponent of the critical space for the stochastic fractional Navier--Stokes equations as established in Section \ref{subsec:scaling}.
        \item The local well-posedness result actually holds for a larger range $\gam\in (\frac 12, \frac{d+2}{3})$. Indeed, for these values of $\gam$ there is a nonempty set of feasible $s$ in Theorem \ref{thm:localWP}.
        \item By inspection of the proof of Theorem \ref{thm:localWP}, one can verify that the result actually holds for noise terms $G(u)$ that grow quadratically. For simplicity, we only consider noise with linear growth, because this is the setting we need for the global well-posedness theory.
    \end{enumerate}
\end{remark}

\subsection{Global well-posedness results}\label{subsec:GWP}
We build on the local well-posedness by proving global well-posedness in the particular setting where $d=2$ and $q$ is sufficiently large relative to $\gamma$. Our proof relies on energy estimates of the vorticity; in fact, we first show global well-posedness of the vorticity form and then construct a global solution of the velocity form through the Biot--Savart operator. In order to obtain a closed-form expression for the vorticity and to identify the velocity via the Biot--Savart operator, we further restrict to noise that is linear and mean-free. To be precise, we consider noise $G(u)=(g_n(\cdot, u))_{n\geq 1}$ given by
\begin{equation} \label{global_noise_thm}g_n(x,u):= \Pi[\alpha_n(x)u]\quad \text{ with }(\alpha_n)_{n\geq 1}\in \ell^2(W^{1,\infty}(\T^2)),\end{equation}
where $\Pi$ is the projection onto the mean-free subspace, see Section \ref{subsec:proofglobal} for more details. \\

The necessity of the vorticity form to obtain energy estimates, which will then imply global well-posedness by a blow-up criterion, is explained in Section \ref{subsec:roadmap}. Furthermore, such estimates can only be shown to hold in 2D as the vorticity form in 3D contains an additional vortex stretching term which will not cancel in these computations. The requirement for a large $q$ relative to $\gamma$ is to ensure an embedding of $L^\zeta$ spaces into the trace space appearing in the blow-up criterion: as $\gamma$ approaches $1/2$, $q$ must approach infinity. A complete statement of the theorem is given below.

\begin{theorem}[Global well-posedness of the stochastic hypoviscous Navier--Stokes equations]\label{thm:globalWPvelocity}
In addition to the conditions of Theorem \ref{thm:localWP}, assume that $2/q< 2\gam-1$ and that $u_0$ is zero-mean with $\nabla \times u_0 \in L^0_{\mc{F}_0}(\Om; \dot{B}_{q,p}^{2/q+1-2\gam}(\T^2;\R))$.
Then the maximal solution $(u,\sigma)$ of \eqref{eq:abstract_eq}, for the choice of $g_n$ given by \eqref{global_noise_thm}, is global.
\end{theorem}

Section \ref{section global WP} is dedicated to the proof of Theorem \ref{thm:globalWPvelocity}, which concludes in Section \ref{subsec:proofglobal}. A more thorough exposition of the strategy is given in Section \ref{subsec:roadmap}.

\section{Local well-posedness}\label{sec:LWP}
In this section, we prove the local well-posedness result of Theorem \ref{thm:localWP}. For this we need to prove estimates on the nonlinearities. Before turning to those estimates we recall the following paraproduct estimate for products of functions in Bessel potential spaces.

\begin{lemma}[{\cite[Proposition 2.1.1]{Ta00}}]\label{lem:paraproduct}
    Let $s\geq 0$ and $q\in(1,\infty)$ be such that 
    \begin{equation*}
        \frac{1}{q} = \frac{1}{q_1}+\frac{1}{q_2}= \frac{1}{\tilde{q}_1}+\frac{1}{\tilde{q}_2},\quad q_2, \tilde{q}_2\in (1,\infty),\, q_1,\tilde{q}_1\in (1,\infty].
    \end{equation*}
    Then
    \begin{equation*}
        \| fg \|_{H^{s,q}(\T^d)}\lesssim \|f\|_{L^{q_1}(\T^d)}\|g\|_{H^{s,q_2}(\T^d)} + \|g\|_{L^{\tilde{q}_1}(\T^d)}\|f\|_{H^{s, \tilde{q}_2}(\T^d)}, 
    \end{equation*}
    for $f\in L^{q_1}(\T^d)\cap H^{s, \tilde{q}_2}(\T^d)$ and $ g\in L^{\tilde{q}_1}(\T^d) \cap H^{s,q_2}(\T^d)$, where the constant is independent of $f$ and $g$.
\end{lemma}
We now prove the required estimates on the nonlinearities.
\begin{lemma} \label{lem:LWP_estimates}
Let $\gam>\half$, $s\in [\gam, 2\gam)$ and suppose that Assumption \ref{ass:X0X1_g} holds. For 
        \begin{equation*}
        |s-1|<\frac dq< 4\gam-s-1\quad \text{ and }\quad \beta: = \frac{1}{4\gam}\big(\frac{d}{q}+s+1\big)\in (\tfrac 12,1),
    \end{equation*}
    and $F$ and $G$ as in \eqref{eq:abstract_eq}, it holds that 
    \begin{align*}
        \|F(u)-F(v)&\|_{X_0}+\|G(u)-G(v)\|_{\gam(\ell^2; X_\half)}\\
       &\; \lesssim \big(1+ \|u\|_{X_\beta}+ \|v\|_{X_\beta}\big)\|u-v\|_{X_\beta},
    \end{align*}
    for all $u,v\in X_1$.
\end{lemma}
\begin{proof}
    From the conditions on $\gam, s$ and $q$ one can verify that $\beta\in(\frac 12,1)$. Using properties of the Helmholtz projection and the $\gam$-radonifying operators \cite[Proposition 9.3.2]{HNVW17} (which also holds for Bessel potential spaces by lifting), we have
    \begin{equation}\label{eq:est1}
            \begin{aligned}
        \|F(u)-&F(v)\|_{X_0}+\|G(u)-G(v)\|_{\gam(\ell^2; X_\half)}\\
        &\lesssim \|u\otimes u-v\otimes v\|_{H^{-s+1,q}} + \|g(\cdot, u)-g(\cdot, v)\|_{H^{-s+\gam,q}(\ell^2(\NN_1; \RRd))}.
    \end{aligned}
    \end{equation}
    
    \textit{Step 1: estimate for the deterministic nonlinearity.} To further estimate the deterministic nonlinearity $F$ in \eqref{eq:est1}, we will distinguish two cases: $s\geq 1$ and $s<1$.
    
    \textit{The case $s\geq 1$.} First consider the term for $F$. By the Sobolev embedding twice, Hölder's inequality and the identity $u\otimes u-v\otimes v= \frac 12 [(u+v)\otimes (u-v)+ (u-v)\otimes (u+v)],$ we obtain
    \begin{align*}
        \|u\otimes u-v\otimes v\|_{H^{-s+1,q}}&\lesssim \|u\otimes u-v\otimes v\|_{L^\lambda}\\
        &\lesssim \|u+ v\|_{L^{2\lambda}}\|u-v\|_{L^{2\lambda}}\\
        & \lesssim(\|u\|_{L^{2\lambda}}+ \|v\|_{L^{2\lambda}})\|u-v\|_{L^{2\lambda}}\\
         &\lesssim(\|u\|_{H^{\theta,q}}+ \|v\|_{H^{\theta,q}})\|u-v\|_{H^{\theta,q}},
    \end{align*}
    where for application of the Sobolev embeddings $L^\lambda\hookrightarrow H^{-s+1,q}$ and $H^{\theta,q}\hookrightarrow L^{2\lambda}$ we need
    \begin{align}\label{eq:value_theta}
        s\geq 1, \quad-\frac{d}{\lambda}=-s+1-\frac{d}{q},\quad \theta\geq 0 \quad\text{ and }\quad \theta - \frac{d}{q} = -\frac{d}{2\lambda}.
    \end{align}
    Note that the second condition determines $\lambda$ and thus the fourth condition gives an expression for $\theta$:
    \begin{equation*}
        \lambda= \frac{dq}{(s-1)q+d}\quad \text{ and }\quad \theta = \frac{d}{q}-\frac{d}{2\lambda} = \frac{d}{2q}-\half(s-1).
    \end{equation*}
    Since $s-1< d/q$ by assumption it follows that $\lambda>1$ (which justifies the application of Hölder's inequality) and $\theta\geq 0$. The desired estimate now follows using $\theta=-s+2\gam\beta$. 

    \textit{The case $s< 1$.} As $s < 1$, one can no longer use a Sobolev embedding $L^\lambda \hookrightarrow H^{-s+1,q}$. Instead, we shall use the paraproduct estimate of Lemma \ref{lem:paraproduct}, obtaining that for any $q_1,q_2 \in (1,\infty)$ with $\frac{1}{q_1} + \frac{1}{q_2} = \frac{1}{q}$,
    $$\|u\otimes u - v\otimes v\|_{H^{-s+1,q}} \lesssim \norm{u+v}_{H^{-s+1,q_1}}\norm{u - v}_{L^{q_2}} + \norm{u-v}_{H^{-s+1,q_1}}\norm{u + v}_{L^{q_2}}.$$
    Our goal is to find the optimal set of parameters $(\theta, q_1, q_2)$ such that
    $$H^{\theta, q} \hookrightarrow H^{-s+1,q_1} \qquad \textnormal{and} \qquad H^{\theta, q} \hookrightarrow L^{q_2}.$$
    Then, we will be able to choose $\beta$ such that $X_{\beta} = \mathbb{H}^{\theta,q}$ to conclude the proof. The optimal choices for the three unknowns $\theta, q_1$ and $q_2$ are determined by three equations: one from the conjugate relation of $q_1$ and $q_2$, and the other two from the two Sobolev embeddings. The equations are
    \begin{align*}
        \frac{1}{q_1} + \frac{1}{q_2} &= \frac{1}{q},\qquad 
q_1 = \frac{dq}{d - \left(\theta - (1-s)\right)q}\qquad \text{ and }\qquad
q_2 = \frac{dq}{d - \theta q}.
    \end{align*}
    The solution to this system is given by
    $$\theta = \frac{d+(1-s)q}{2q}, \qquad q_1 = \frac{2dq}{d+(1-s)q} \qquad\text{ and }\qquad  q_2 = \frac{2dq}{d-(1-s)q}.$$
    Note that we also require $0 < \theta - (1-s)$ for the Sobolev embedding of $H^{\theta,q} \hookrightarrow H^{-s+1,q_1}$ to be valid, however this can be rewritten as the condition $q < \frac{d}{1-s}$ which is assumed.\\

    Note that we arrive at the same $\theta$ as in the previous case $s \geq 1$. For a consistency check on the parameters $q_1$ and $q_2$, in the intermediate case $s=1$ we would obtain that $q_1 = q_2 = 2q$ which agrees with the $2\lambda$ from the previous case, and of course corresponds to a simple H\"{o}lder inequality. Beyond facilitating the choice of $\theta$, the values of $q_1$ and $q_2$ are not relevant here, and we choose the same $\beta$ as in the previous case and for the statement of the lemma, prescribed by the relation $\theta=-s+2\gam\beta$. Therefore,
    \begin{align*}
        \|u\otimes u-v\otimes v\|_{H^{-s+1,q}} &\lesssim \norm{u+v}_{H^{-s+1,q_1}}\norm{u - v}_{L^{q_2}} + \norm{u-v}_{H^{-s+1,q_1}}\norm{u + v}_{L^{q_2}}\\
        &\lesssim \norm{u+v}_{H^{\theta,q}}\norm{u - v}_{H^{\theta,q}}\\
        &\lesssim \left(\norm{u}_{H^{\theta,q}} + \norm{v}_{H^{\theta,q}}\right)\norm{u - v}_{H^{\theta,q}}\\
        &= (\norm{u}_{X_{\beta}} + \norm{v}_{X_{\beta}})\norm{u - v}_{X_{\beta}}
    \end{align*}
    as required.

\textit{Step 2: estimate for the stochastic nonlinearity.} It remains to further estimate the stochastic nonlinearity $G$ in \eqref{eq:est1}. Using $s\geq \gam$ and the Lipschitz condition on $g$ (see Assumption \ref{ass:X0X1_g}), we obtain
    \begin{align*}
        \|g(\cdot, u)-g(\cdot, v)\|_{H^{-s+\gam,q}(\ell^2)}&\lesssim \|g(\cdot, u)-g(\cdot, v)\|_{L^q(\ell^2)}\\
        &\lesssim \|\, |u-v|\, \|_{L^q}\\
        &\lesssim \|u-v\|_{H^{\theta, q}},
    \end{align*}
    for all $\theta\geq 0$. In particular, we can take $\theta$ as in \eqref{eq:value_theta} to identify the $H^{\theta, q}$-norm with the $X_\beta$-norm.
\end{proof}

To apply the abstract theory of stochastic maximal regularity in \cite{AV22, AV22_part2, AV25_survey}, we need the following property of the linear deterministic operator. In particular, the boundedness of the $\Hinf$-calculus for linear differential operators provides a sufficient condition for stochastic maximal regularity for linear parabolic stochastic equations. For more details on the $\Hinf$-calculus, we refer to \cite[Chapter 10]{HNVW17}.
\begin{lemma}\label{lem:calculus}
Let $\gam>0$, $s\in\RR$ and $q\in(1,\infty)$. Then $(-\del)^\gam$ on $H^{-s,q}(\T^d;\RRd)$ with $D((-\del)^\gam):= H^{-s+2\gam,q}(\T^d;\RRd)$ has a bounded $\Hinf$-calculus of angle $\om_{\Hinf}((-\del)^\gam)=0$.
\end{lemma}
\begin{proof} First, let $\gam=1$. Then the result follows from the case $s=0$ and a lifting argument. For details we refer to, e.g., \cite[Lemma 2.6]{LLRV24} where this is proved on $\RR^d$, and the case $\T^d$ can be proved similarly. Finally, for general  $\gam>0$ the result follows by \cite[Proposition 15.2.11]{HNVW24}.
\end{proof}

We proceed with the proof of our main local well-posedness result in Theorem \ref{thm:localWP}. The proof contains the following steps:
\begin{enumerate}[1.]
    \item  Application of \cite[Theorem 4.7]{AV25_survey} to obtain a unique maximal solution.
    \item Regularisation in time via \cite[Theorems 5.6 \& 5.7]{AV25_survey}.
    \item Bootstrapping higher-order integrability in space.
    \item Bootstrapping higher-order smoothness in space.
\end{enumerate}
\begin{proof}[Proof of Theorem \ref{thm:localWP}] 
\textit{Step 1: local well-posedness.} Lemma \ref{lem:calculus} and \cite[Theorem 3.14]{AV25_survey} imply that $(A, 0)\in \SMR_{p,\kappa}^{\bullet}$ for all $p\in(2,\infty)$ and $\kappa\in [0, \frac{p}{2}-1)$. It remains to check \cite[Assumption 4.1]{AV25_survey}. From Lemma \ref{lem:LWP_estimates} it follows that the criticality condition (see \cite[Equation (4.3)]{AV25_survey} with $m=1$, $\rho=1$ and $\beta$ as in Lemma \ref{lem:LWP_estimates}) is satisfied if
\begin{equation}\label{eq:LWP_critical}
    \frac{1+\kappa}{p}\leq 2(1-\beta) = 2- \frac{1}{2\gam}\big(\frac{d}{q}+s+1\big)
\end{equation}
Note that with our choice $\kappa:=-1+\frac{p}{2}(4-\gam^{-1}(\frac{d}{q}+s+1))$ we obtain equality in \eqref{eq:LWP_critical}. Furthermore, the conditions $3\gam -s-1 < \frac dq$ and $\frac{2}{p}+\frac{d}{q\gam}\leq 4-\frac{s+1}{\gam}$ ensure that $$4-\gam^{-1}\big(\frac{d}{q}+s+1\big)\in [2/p, 1).$$ This yields $\kappa\in [0, \frac{p}{2}-1)$.
Now, \cite[Theorem 4.7]{AV25_survey} gives a unique maximal solution $(u, \sigma)$ to \eqref{eq:abstract_eq} with $\sigma>0$ a.s. and the space of initial data given by
\begin{equation}\label{eq:tracespace}
    X_{1-\frac{1+\kappa}{p},p}= \BB^{-s+ 2\gam (1-\frac{1+\kappa}{p})}_{q,p} = \BB^{\frac{d}{q}+1-2\gam}_{q,p},
\end{equation}
where the last equality follows from the choice of $\kappa$.

\textit{Step 2: regularisation in time.} Since $(A, 0)\in \SMR_{p,\kappa}^{\bullet}$ for all $p\in(2,\infty)$ and $\kappa\in [0, \frac{p}{2}-1)$, it follows from \cite[Theorem 5.6]{AV25_survey} (for $\kappa>0$) or \cite[Theorem 5.7]{AV25_survey} (for $\kappa=0$) that the regularity of the solution $u$ obtained in Step 1, improves to
\begin{equation*}
    u\in H^{\theta,r}_\loc((0,\sigma); \Hs^{-s+2\gam(1-\theta),q})\cap C^{\theta -\eps}_\loc((0,\sigma); \Hs^{-s+2\gam(1-\theta),q} ),\quad \theta\in (0,\tfrac{1}{2}), r\in[2,\infty), \eps\in(0,\theta).
\end{equation*}

\textit{Step 3: bootstrapping space integrability.} In this step we will show that 
\begin{equation}\label{eq:resultStep3}
    u\in \bigcap_{r,\zeta\in(2,\infty)}\bigcap_{\lambda\in[0,\half)}H^{\lambda, r}_\loc
((0,\sigma); \Hs^{-s+2\gam(1-\lambda),\zeta}(\T^d;\RRd))\, \text{ a.s.}
\end{equation}
By bootstrapping it suffices to show that there exists a $\delta>0$ such that for every $\zeta\geq q$ the following implication holds
\begin{equation}\label{eq:bootstrap_space_int}
\begin{aligned}
        u\in \bigcap_{r\in(2,\infty)}&\bigcap_{\lambda\in[0,\half)}H^{\lambda, r}_\loc
((0,\sigma); \Hs^{-s+2\gam(1-\lambda),\zeta}) \,\text{ a.s. }\\
&\implies u\in \bigcap_{r\in(2,\infty)}\bigcap_{\lambda\in[0,\half)}H^{\lambda, r}_\loc
((0,\sigma); \Hs^{-s+2\gam(1-\lambda),\zeta+\delta})\,\text{ a.s.}
\end{aligned}
\end{equation}
Let $(\sigma_n)_{n\geq1}$ be a localising sequence for $(u,\sigma)$ and take $\zeta\geq q$. Furthermore, let $t_1\in(0,\infty)$. On the set $\{\sigma_n>t_1\}\times (t_1, \sigma_n)$ we claim that the following estimate holds
\begin{align}\label{eq:claim_estFG}
    \|F(u)\|_{\Hs^{-s,\zeta+\delta}}+ \|G(u)\|_{\Hs^{-s+\gam, \zeta+\delta}(\ell^2)}&\lesssim 1+ \|u\|^2_{\Hs^{-s+2\gam,\zeta}},
\end{align}
for a suitable $\delta$ independent of $\zeta$. To prove the claim, we consider the deterministic and stochastic nonlinearities separately.

\textit{Estimate for $F$ with $s\geq 1$.} If $s\geq 1$, then by Sobolev embedding twice and Hölder's inequality, we obtain
\begin{align*}
     \|F(u)\|_{\Hs^{-s,\zeta+\delta}}&\lesssim \|u\otimes u \|_{H^{-s+1,\zeta+\delta}}
     \lesssim \|u\otimes u \|_{L^\lambda}\\
     &\lesssim \|u\|^2_{L^{2\lambda}}\lesssim \|u\|^2_{H^{-s+2\gam,\zeta}},
\end{align*}
where for application of the Sobolev embeddings $L^\lambda\hookrightarrow H^{-s+1, \zeta+\delta}$ and $H^{-s+2\gam, \zeta}\hookrightarrow L^{2\lambda}$ we need
\begin{equation*}
    s\geq 1, \quad -\frac{d}{\lambda}=-s+1-\frac{d}{\zeta+\delta},\quad  -s+2\gam\geq0\quad \text{ and }\quad -s+2\gam-\frac{d}{\zeta}\geq -\frac{d}{2\lambda}
\end{equation*}
The second condition determines $\lambda= \frac{d(\zeta+\delta)}{(s-1)(\zeta+\delta)+d}$ and thus the fourth condition can be rewritten as
\begin{equation}\label{eq:fourthcondition}
    -s+2\gam-\frac{d}{\zeta}\geq -\half(s-1)-\frac{d}{2(\zeta+\delta)}\quad\iff\quad 4\gam-(s+1)-\frac{2d}{\zeta}\geq -\frac{d}{\zeta+\delta}.
\end{equation}
Set $\delta := \min\{\frac{q^2}{d}(4\gam-(s+1)-\frac{d}{q}), \frac{1}{4d}\}$ and note that $\delta>0 $ since $\frac{d}{q}< 4\gam-(s+1)$ by assumption. 
Using $\zeta\geq q$ and the definition of $\delta$, yields
\begin{align*}
    4\gam-(s+1)-\frac{d}{\zeta}&\geq 4\gam-(s+1)-\frac{d}{q} \geq \frac{d\delta}{q^2}\geq \frac{d\delta}{\zeta(\zeta+\delta)}= -\frac{d}{\zeta+\delta}+\frac{d}{\zeta}.
\end{align*}
Hence, the condition in \eqref{eq:fourthcondition} is satisfied. 

\textit{Estimate for $F$ with $s < 1$.} If $s<1$, then the paraproduct estimate in Lemma \ref{lem:paraproduct} implies
\begin{equation*}
    \|F(u)\|_{H^{-s, \zeta+\delta}}\lesssim \|u\otimes u \|_{H^{-s+1, \zeta+\delta}} \lesssim \|u\|_{L^{r_1}}\|u\|_{H^{-s+1, r_2}},
\end{equation*}
where $\frac 1{r_1}+ \frac 1{r_2}=\frac 1{\zeta+\delta}$. Set $\delta:=\min\{\frac{q^2}{2d}(2\gam-1),\frac{q^2}{2d}( 4\gam-s-1-\frac dq), \frac 1{4d}\}>0$. Then for any $\zeta\geq q$, one can verify that
\begin{equation*}
    \max\Big\{\frac 1\zeta -\frac{2\gam-s}{d}, 0\Big\}+ \max\Big\{\frac 1\zeta -\frac{2\gam-1}{d}, 0\Big\}< \frac{1}{\zeta+\delta}.
\end{equation*}
Hence, we can choose $r_1, r_2> 1$ such that
\begin{equation*}
    \frac 1{r_1}+ \frac 1{r_2}=\frac 1{\zeta+\delta}, \quad \frac 1{r_1}> \max\Big\{\frac 1\zeta -\frac{2\gam-s}{d}, 0\Big\}\quad \text{ and }\quad \frac 1{r_2}>\max\Big\{\frac 1\zeta -\frac{2\gam-1}{d}, 0\Big\}.
\end{equation*}
Therefore, the Sobolev embeddings $H^{-s+2\gam, \zeta} \hookrightarrow L^{r_1}$ and $H^{-s+2\gam,\zeta}\hookrightarrow H^{-s+1,r_2}$ apply and this proves the required estimate.

\textit{Estimate for $G$.} Using the assumption for $g$ in Assumption \ref{ass:X0X1_g} one can prove with an argument similar to the one above for $F$ that 
\begin{equation*}
    \|G(u)\|_{\Hs^{-s+\gam, \zeta+\delta}(\ell^2)}\lesssim 1+\|u\|_{\Hs^{-s+2\gam,\zeta}}.
\end{equation*}
Combining all the estimates above completes the proof of the claim \eqref{eq:claim_estFG}.\\

Note that by the assumption in \eqref{eq:bootstrap_space_int}, the trace embedding \cite[Proposition 2.1(2)]{AV25_survey} and the Sobolev embedding, we have 
\begin{align}\label{eq:IC1}
    \ind_{\{\sigma>t_1\}} u(t_1)\in X_{1-\frac{1}{r},r}=\BB^{-s+2\gam(1-\frac{1}{r})}_{\zeta,r}\subseteq \BB^{-s+2\gam(1-\frac{1+\tilde{\kappa}}{r})}_{\zeta+\delta,r},
\end{align}
where for application of the Sobolev embedding we need $-\frac{d}{\zeta}\geq - \frac{2\gam\tilde{\kappa}}{r}-\frac{d}{\zeta+\delta}$. Set $\tilde{\kappa}=r/4$. Then using $\gam>\half$, $\delta\leq \frac{1}{4d}$ and $\zeta(\zeta+\delta)\geq q^2>1$, gives
\begin{equation*}
    \frac{2\gam\tilde{\kappa}}{r}\geq \frac{\tilde{\kappa}}{r}=\frac{1}{4}\geq d\delta \geq \frac{d\delta}{\zeta(\zeta+\delta)} =-\frac{d}{\zeta+\delta}+ \frac{d}{\zeta},
\end{equation*}
which thus justifies the application of the Sobolev embedding for $\tilde{\kappa}=r/4$. Note that $\tilde{\kappa}<\frac r2 -1 $ only if $r>4$. So to apply stochastic maximal regularity, we have to choose $R>\max\{r,4\}$ and set $\tilde{\kappa}=R/4< R/2-1$, then \eqref{eq:IC1} is also valid with $r$ replaced by $R$. 
Since $(A,0)\in \SMR_{R,\tilde{\kappa}}$ for $Y_0=\Hs^{-s, \zeta+\delta}$ and $Y_1=\Hs^{-s+2\gam,\zeta+\delta}$, the implication \eqref{eq:bootstrap_space_int} follows from stochastic maximal $L^R_{\tilde{\kappa}}$-regularity in \cite[Proposition 3.11]{AV25_survey} with inhomogeneities $F(u)$ and $G(u)$. If $r>2$, then the result follows with standard Sobolev embeddings.

\textit{Step 4: bootstrapping spatial smoothness.} 
We have by Step 3 that a.s. on $\{\sigma_n>t_1\}$
\begin{equation}\label{eq:Step4_reg_u}
    u \in C([t_1,\sigma_n]; \Hs^{-s+2\gam-\eps,\zeta}(\T^d;\RRd)),
\end{equation}  
for all $\eps>0$ and $\zeta\in(2,\infty)$. Indeed, this follows from \eqref{eq:resultStep3} and the Sobolev embedding from $H^{\lambda, r}$ into $C^{\lambda-\frac{1}{r}}$ (see \cite[Theorem 14.7.9, Proposition 14.6.8 \& Corollary 14.4.27]{HNVW24}).

To continue, we will estimate the nonlinearities $F(u)$ and $G(u)$ on $\{\sigma_n>t_1\}\times \Om$ using the obtained regularity of $u$ in \eqref{eq:Step4_reg_u}. Below we assume that $\gam>  (s+1)/3$. For the other case $\half <\gam \leq (s+1)/3$ it suffices to iterate the argument below. Fix $\eps>0$ small enough and $\zeta\in (2,\infty)$ large enough such that 
\begin{equation*}
    -s+2\gam-\eps-\frac{d}{\zeta}>0\quad \text{ and }\quad -s+2\gamma-\eps> -\gam+1.
\end{equation*}
Note that we can always choose $\eps$ and $\zeta$ such that the first condition is satisfied and since $\gam> (s+1)/3$ we can also choose $\eps$ small enough such that the second condition holds. With the choice of these parameters, we have the embeddings
\begin{equation}\label{eq:SobembStep4}
    \Hs^{-s+2\gam-\eps, \zeta}(\T^d;\RRd)\hookrightarrow L^{\infty}(\T^d;\RRd)\quad \text{ and }\quad \Hs^{-s+2\gam-\eps,\zeta}(\T^d;\RRd)\hookrightarrow \Hs^{-\gam+1,\zeta}(\T^d;\RRd).
\end{equation}
By Lemma \ref{lem:paraproduct} (using that $\gam\leq 1$) and \eqref{eq:SobembStep4}, we obtain
\begin{equation*}
    \|F(u)\|_{\Hs^{-\gam,\zeta}}\lesssim \|u\otimes u\|_{H^{-\gam+1,\zeta}}\lesssim \|u\|_{L^\infty}\|u\|_{H^{-\gam+1,\zeta}} \lesssim \|u\|^2_{\Hs^{-s+2\gam-\eps,\zeta}}.
\end{equation*}
The right-hand side is finite by \eqref{eq:Step4_reg_u}. This proves that
\begin{equation*}
    F(u)\in \bigcap_{\zeta\in (2,\infty)} L^\infty(t_1, \sigma_n; \Hs^{-\gam,\zeta})\quad \text{a.s. on }\{\sigma_n >t_1\}.
\end{equation*}
Using the assumptions on $g$ (see Assumption \ref{ass:X0X1_g}) and \eqref{eq:SobembStep4}, we obtain that 
\begin{equation*}
    \|G(u)\|_{\Ls^\zeta(\ell^2)}\lesssim \|g(\cdot, u)\|_{L^\zeta(\ell^2)}\lesssim  1+\|u\|_{L^\infty}
\lesssim 1+ \|u\|_{\Hs^{-s+2\gam-\eps,\zeta}} ,
\end{equation*}
where the right-hand side is finite by \eqref{eq:Step4_reg_u}. This proves that
\begin{equation*}
    G(u)\in \bigcap_{\zeta\in (2,\infty)} L^\infty(t_1, \sigma_n; \Ls^{\zeta}(\T^d;\ell^2(\NN_1;\RRd)))\quad \text{a.s. on }\{\sigma_n >t_1\}.
\end{equation*}
For the initial value we find with \eqref{eq:Step4_reg_u} that
\begin{equation*}
    \ind_{\{\sigma>t_1\}} u(t_1)\in \Hs^{-s+2\gam-\eps,\zeta}\subseteq \mathbb{B}_{\zeta,r}^{-\gam + 2\gam (1-\frac{1+\tilde{\kappa}}{r})},
\end{equation*}
for all $r\in(2,\infty)$ and $\tilde{\kappa}\in [0, \frac{r}{2}-1)$ large enough. Indeed, for $\frac{1+\tilde{\kappa}}{r}$ sufficiently close to $\half$, we find that $-\gam+2\gam(1-\frac{1+\tilde{\kappa}}{r})$ is arbitrarily close to zero and thus $-s+2\gam-\eps > -\gam+2\gam(1-\frac{1+\tilde{\kappa}}{r})$ since $-s+2\gam-\eps>0$. This embedding of the Bessel potential space into the Besov space then follows from \cite[Theorems 14.7.9, 14.6.14(i) \& Proposition 14.6.8]{HNVW24}.

Since $(A,0)\in \SMR_{r,\tilde{\kappa}}$ for $Y_0=\Hs^{-\gam, \zeta}$ and $Y_1=\Hs^{\gam,\zeta}$, the desired regularity
\begin{equation*}
    u\in L^r_\loc((0,\sigma); \Hs^{\gam,\zeta})\cap C^{\theta -\eps}_\loc((0,\sigma); \Hs^{-\gam+2\gam(1-\theta),\zeta}),
\end{equation*}
with $\theta \in (0,\half)$, $\eps\in (0,\theta)$,  $r,\zeta\in (2,\infty)$ follows from \cite[Proposition 3.11]{AV25_survey} and Sobolev embeddings.
\end{proof}

\section{Global well-posedness in two dimensions} \label{section global WP}
In this section, we prove global well-posedness of the stochastic hypoviscous Navier--Stokes equations on $\T^2$. First, in Section \ref{subsec:roadmap} we will give an overview of the methods used in the rest of this section to prove  global well-posedness. This entails abstract blow-up criteria (Section \ref{subsec:blowupcrit}), a priori estimates for the vorticity equation (Section \ref{subsec:vorticity}), and using the Biot--Savart operator to finally prove the global well-posedness result of Theorem \ref{thm:globalWPvelocity} (Section \ref{subsec:proofglobal}).

\subsection{Roadmap to global well-posedness}\label{subsec:roadmap}

We explain the key ideas behind our approach to proving global well-posedness for the stochastic hypoviscous Navier--Stokes equations in two dimensions. 

\subsubsection{Blow-up criteria} Relying on the theory for nonlinear stochastic equations from \cite{AV25_survey}, it suffices to verify certain blow-up criteria to obtain global well-posedness. We refer to \cite[Section 5]{AV25_survey} for an elaborate discussion on blow-up criteria. To verify such blow-up criteria, one typically has to prove certain energy estimates for the local maximal solution $(u, \sigma)$. 

In Section \ref{subsec:blowupcrit} we collect the required blow-up criteria and prove that these are actually independent of the parameters $p,q$ and $s$ using the regularisation results from Section \ref{sec:LWP}. This shows that, to establish
global well-posedness, it suffices to verify that for every $t_0 > 0$, the maximal solution $(u, \sigma)$ satisfies 
\begin{equation*}
    \sup_{t\in [t_0, \sigma)} \|u(t)\|_{\mathbb{B}^{\beta_0}_{q_1, \infty}}<\infty
\end{equation*}
whenever $ t_0<\sigma<T$, for some $q_1>q$ large enough and $\beta_0:= \frac{d}{q}+1-2\gam$.  

To derive estimates in $\BB^{d/q+1-2\gam}_{q_1,\infty}$, it is natural to look for sufficiently large $\zeta$ such that $L^\zeta \subseteq B^{d/q+1-2\gam}_{q_1,\infty}$ and work in the more favourable $L^\zeta$ space. Note that when the Sobolev index $d/q+1-2\gam$ is negative, we necessarily have that
$L^{q_1} \subseteq B^{d/q+1-2\gam}_{q_1,\infty} $,
hence $L^\zeta \subseteq B^{d/q+1-2\gam}_{q_1,\infty}$ for any $q_1 \leq \zeta$. The condition $d/q+1-2\gam<0$ is equivalent to $d/q< 2\gam-1$, which corresponds to the limiting case $s = 2\gamma$ in the conditions of Theorem \ref{thm:localWP}. To obtain a non-empty set of parameters, we need to impose the condition
\begin{equation*}
    \max\{|s-1|, 3\gam-s-1\}< 2\gam-1.
\end{equation*}
Obviously, the condition $3\gam-s-1<2\gam -1$ leads to $s>\gam$, thus removing the endpoint from the condition on $s$ in Theorem \ref{thm:localWP}. On the other hand, the condition $|s-1|<2\gam-1$ is equivalent to $2-2\gam<s<2\gam$. Hence, we are left with the following intervals for $s$:
\begin{equation*}
    \begin{cases}
        2-2\gam<s<2\gam &\mbox{ if } \half < \gam< \frac 23,\\
        \gam<s<2\gam&\mbox{ if } \gam\geq \frac 23.
    \end{cases}
\end{equation*}


\subsubsection{Energy estimates for the vorticity equation}
We have established a non-empty set of parameters for which the local well-posedness of Theorem \ref{thm:localWP} applies, as does the embedding $L^{\zeta} \subseteq B^{d/q+1-2\gam}_{q_1,\infty}$. Thus, to demonstrate non-explosion of the local solution, we can try to prove $L^\zeta$-estimates. Unfortunately after all of this motivation for considering $L^\zeta$-estimates, we do not have a clear way of obtaining them for the equation \eqref{eq:abstract_eq} directly. The issue is, unsurprisingly, the nonlinear term; applying an It\^{o} formula for the $L^\zeta$-norm, we must deal with a term
$$\big\langle\P (u \cdot \nabla)u,u\abs{u}^{\zeta-2}\big\rangle_{L^2}$$
where $\abs{u}$ is the Euclidean norm of $u$. Whilst we enjoy a cancellation without the Helmholtz projection,
$$\big\langle(u \cdot \nabla)u, u\abs{u}^{\zeta-2}\big\rangle_{L^2} = 0,$$
or observe a cancellation in the case $\zeta=2$,
$$\inner{\P (u \cdot \nabla)u}{u}_{L^2} = \inner{ (u \cdot \nabla)u}{u}_{L^2} = 0,$$
we do not in general have that $\P$ is the identity on $u\abs{u}^{\zeta-2}$ which prevents us from obtaining a cancellation and closing the estimates. In two dimensions however, one can instead work with the vorticity form of the equation where the Helmholtz projection is not present. Applying the curl operator $\nabla \times$, which on a two-dimensional vector field $f$ produces a scalar defined by $$ \nabla \times f = \partial_1f^2 - \partial_2f^1,$$
gives, at least formally,
the following vorticity form of equation \eqref{eq:abstract_eq}:
\begin{equation} \label{eq: firstvorticity} \dd \xi + (-\del)^\gam \xi \dd t  = -(u \cdot \grad)\xi \dd t + \nabla \times \left[ G(u)\right] \dd W_{\ell^2}\end{equation}
where $\xi = \nabla \times u$, see for example \cite[Chapter 12]{robinson2016three}. In this form, the nonlinear term does not present any issue since
$$\big\langle (u \cdot \nabla)\xi,\xi\abs{\xi}^{\zeta-2}\big\rangle_{L^2} = 0$$
so our hopes for non-explosion lie in energy estimates of the vorticity form.

\subsubsection{Global well-posedness for the velocity equation}
Still, it is not clear how we can use the vorticity estimates in the context of the unique maximal solution $(u,\sigma)$ to \eqref{eq:abstract_eq}. Even with the instantaneous regularisation established in Theorem \ref{thm:localWP}, the solution is not regular enough for $\xi \in L^q$. Therefore we will not simply be taking the curl of $u$, but rather working with the vorticity equation from the ground up and translating the regularity to $u$. To this end, we need to establish a closed-form vorticity equation. One obstruction in equation \eqref{eq: firstvorticity} is the structure of the noise. For a clean expression, we assume that the operators $g_n$ comprising $G$ are linear of the form
$$g_n(x,u) = \alpha_n(x)u\quad \text{ with }(\alpha_n)_{n\geq 1}\in \ell^2(W^{1,\infty}(\T^2)).$$
Then we can explicitly compute
\begin{equation} \label{explicitly compute}\nabla \times \left[ \P g_n(\cdot,u)\right] = \alpha_n\xi + (\nabla \alpha_n) \times u.\end{equation}
Furthermore we need to prescribe $u$ as a function of $\xi$, where we recall that formally $\xi=\grad \times u$. 
 This can be achieved through the \emph{Biot--Savart operator} which we denote $\BS$. We refer to \cite{bertozzi2002vorticity} for more details on this operator, which acts as an inverse to the curl. This leads to a closed-form expression of the vorticity equation
 \begin{equation*} 
\dd \xi + (-\del)^\gam \xi \dd t  = -(\BS\xi \cdot \grad)\xi \dd t + \sum_{n \geq 1} \left(\alpha_n\xi + (\nabla \alpha_n) \times \BS\xi\right)\dd W^n_t.\end{equation*}
We prove global well-posedness for this
equation and then a global solution of the stochastic hypoviscous Navier--Stokes equation in velocity form is constructed using $u=\BS \xi$.

\subsection{Blow-up criteria}\label{subsec:blowupcrit}

As an immediate consequence of \cite[Theorems 5.1 \& 5.2]{AV25_survey}, we get the following criteria to verify global well-posedness. We note that the blow-up criteria in this section actually hold for arbitrary dimensions, while for verifying the blow-up criteria later on we restrict ourselves to $d=2$.

\begin{theorem}[Blow-up criteria for the critical setting]\label{thm:blow-up_crit_critical}
Suppose that the conditions of Theorem \ref{thm:localWP} hold, and let $(u, \sigma)$ be the $L^p_\kappa$-maximal solution to \eqref{eq:abstract_eq}. Then
\begin{enumerate}
    \item[(1)] $\mathbf{P}(\sigma < \infty, \lim\limits_{t \uparrow \sigma} u(t) \text{ \textit{exists in} } X_{1 - \frac{1+\kappa}{p}, p}) = 0;$
    \item[(2)] $\mathbf{P}(\sigma < \infty, \sup\limits_{t \in [0, \sigma)} \|u(t)\|_{X_{1 - \frac{1+\kappa}{p}, p}} + \|u\|_{L^p(0, \sigma; X_{1 - \frac{\kappa}{p}})} < \infty) = 0.$
\end{enumerate}
\end{theorem} 

\begin{theorem}[Blow-up criteria for the non-critical setting]\label{thm:blow-up_crit_noncritical}
Suppose that the conditions of Theorem \ref{thm:localWP} hold (with non-critical $\kappa$), and let $(u, \sigma)$ be the $L^p_\kappa$-maximal solution to \eqref{eq:abstract_eq}. Then
\begin{equation*}
    \mathbf{P}(\sigma < \infty,\sup\limits_{t \in [0, \sigma)} \|u(t)\|_{X_{1 - \frac{1+\kappa}{p}, p}}<\infty) = 0,
\end{equation*}
if $(p,\kappa)$ is non-critical

\end{theorem}

Using the instantaneous regularisation of the local maximal solution in Theorem \ref{thm:localWP}, we show that the blow-up criteria formulated above are independent of the parameters $p,q, s$ and $\kappa$.  
\begin{theorem}[Parameter independence of the blow-up criteria]\label{thm:indep_parameter_blow-up}
Suppose that the conditions of Theorem \ref{thm:localWP} hold, and let $(u, \sigma)$ be the $(p,q,\kappa_c, s)$-solution to \eqref{eq:abstract_eq}, where $\kappa_c:=-1 + \frac{p}{2}\big(4-\gam^{-1}(\frac{d}{q}+s+1)\big)$. Suppose that $p_0\in (2,\infty)$, $q_0\in (2,\infty)$ and $s_0\in [\gam, 2\gam)$ are such that the assumptions of Theorem \ref{thm:localWP} hold. Set
\begin{equation*}
    \beta_0:=\frac{d}{q_0}+1-2\gam\quad \text{ and }\quad \tilde{s}_0:=\frac{d}{q_0}+1-2\gam\Big(1-\frac{1}{p_0}\Big).
\end{equation*}
Then for all $0<t_0<T<\infty$, we have
\begin{enumerate}[(1)]
    \item\label{it:1} $\mathbf{P}\big(t_0<\sigma < T,\sup\limits_{t \in [t_0, \sigma)} \|u(t)\|_{{\BB^{\beta_0}_{q_1,\infty}(\T^d;\RR^d)}}<\infty\big) = 0$ for all $q_1> q_0$,
    \item\label{it:2} $\mathbf{P}\big(t_0<\sigma < T,\sup\limits_{t \in [t_0, \sigma)} \|u(t)\|_{{\BB^{\beta_0}_{q_0,p_0}(\T^d;\RR^d)}}+\|u\|_{L^{p_0}(t_0,\sigma; H^{\tilde{s}_0,q_0}(\T^d;\RR^d))}<\infty\big) = 0.$
\end{enumerate}
\end{theorem}

\begin{proof}
    We follow the proof of \cite[Theorem 2.10]{AV23}.

    \textit{Proof of \ref{it:2}}. We start with some preparations. Fix $0<t_0<T<\infty$ and let $\kappa_c:=-1 + \frac{p}{2}\big(4-\gam^{-1}(\frac{d}{q}+s+1)\big)$ as in Theorem \ref{thm:localWP}. In addition, let $(u, \sigma)$ be the $(p,q, \kappa_c,s)$-solution to \eqref{eq:abstract_eq} provided by Theorem \ref{thm:localWP}. Then Theorem \ref{thm:blow-up_crit_critical}(2) together with \eqref{eq:tracespace} implies that
    \begin{equation}\label{eq:para_indep1}
        \mathbf{P}\big(\sigma<\infty,\sup\limits_{t \in [0, \sigma)} \|u(t)\|_{\BB^{\frac{d}{q}+1-2\gam}_{q,p}(\T^d; \RR^d)} +  \|u\|_{L^p(0, \sigma; \Hs^{\tilde{s},q}(\T^d; \RR^d))} < \infty\big)=0,
    \end{equation}
    where, using the expression for $\kappa_c$, we have
    \begin{equation*}
       \tilde{s}:= -s+2\gam \big(1-\frac{\kappa_c}{p}\big)= \frac{d}{q}+1 -2\gam\big(1-\frac{1}{p}\big).
    \end{equation*}
    Let $\theta_c:= \frac{\kappa_c}{p} < \half - \frac{1}{p}$, then by Theorem \ref{thm:localWP} and weighted Sobolev embeddings (\cite[Proposition 2.7]{AV22}), we find
    \begin{equation}\label{eq:u_instan}
        u\in H^{\theta_c, p}_{\loc}([0,\sigma), w_{\kappa_c}; \Hs^{-s+2\gam(1-\theta_c),q}(\T^d; \RR^d))\hookrightarrow L^p_{\loc}([0,\sigma); \Hs^{ \tilde{s}, q}(\T^d; \RR^d)).
    \end{equation}
    Indeed, we can apply the weighted Sobolev embedding since by definition of $\theta_c$, we have
    \begin{equation*}
        \theta_c-\frac{\kappa_c+1}{p}\geq -\frac{1}{p}. 
    \end{equation*}
By the instantaneous regularisation in Theorem \ref{thm:localWP} and the progressive measurability of $u$, we have for any $\lambda\in (0,\gam)$ that
\begin{equation}\label{eq:para_indep2}
    \ind_{\{\sigma>t_0\}}u(t_0)\in L^0_{\mathcal{F}_{t_0}}(\Omega; C^{\lambda }(\T^d; \RR^d)).
\end{equation}
Let $p_0\in (2,\infty)$, $q_0\in (2,\infty)$ and $s_0\in [\gam, 2\gam)$ be such that the assumptions of Theorem \ref{thm:localWP} hold with these parameters. In addition, define the critical weight
\begin{equation}\label{eq:kappa0}
    \kappa_0:= \kappa_{c,0}:= -1  + \frac{p_0}{2}\big(4-\gam^{-1}\big(\frac{d}{q_0}+s_0+1\big)\big),
\end{equation}
and the critical space for this setting is $\BB^{\beta_0}_{q_0,p_0}(\T^d;\RRd)$ with $\beta_0:=\frac{d}{q_0}+1-2\gam$. In particular, we note that by assumption $q_0> \frac{d}{4\gam-s_0-1}$ (see Theorem \ref{thm:localWP}) and $s_0\geq \gam$, hence it follows that
$\frac{d}{q_0}< 3\gam-1$ and 
\begin{equation}\label{eq:beta0}
    \beta_0 = \frac{d}{q_0}+1-2\gam< (3\gam-1)+1-2\gam = \gam.
\end{equation}
Because $\beta_0<\gam$, we can take $\lambda\in (\max\{0, \beta_0\}, \gam)$ and for such $\lambda$ it holds that $C^\lambda=B^\lambda_{\infty,\infty}\hookrightarrow B^{\beta_0}_{q_0, p_0}$. Therefore, \eqref{eq:para_indep2} implies 
\begin{equation}\label{eq:para_indep3}
    \ind_{\{\sigma>t_0\}}u(t_0)\in L^0_{\mathcal{F}_{t_0}}(\Omega; \BB^{\beta_0}_{q_0, p_0}(\T^d; \RR^d)).
\end{equation}
Up to a shift of the time variable, \eqref{eq:para_indep3} and the local well-posedness from Theorem \ref{thm:localWP} imply the existence of a $(p_0,q_0, \kappa_0, s_0)$-maximal solution $(v, \tau)$ on $[t_0, \infty)$ of \eqref{eq:abstract_eq}, i.e.,
    \begin{equation}\label{eq:v}
       \left\{
        \begin{aligned}
            &\dd v + (-\del)^\gam v \dd t = -\P [\grad \cdot (v\otimes v)]\dd t + (\P[g_n(\cdot, v)])_{n\geq1}\dd W_{\ell^2}&\text{ for }&t\geq t_0,\\
            &v|_{t=t_0}=\ind_{\{\sigma>t_0\}}u(t_0),
        \end{aligned}
        \right.
\end{equation}
Moreover, this solution $(v,\tau)$ instantaneously regularises in time and space:
     \begin{equation}\label{eq:instan_v}
    v\in H^{\theta,r}_\loc((t_0,\tau); \Hs^{-\gam+2\gam(1-\theta),\zeta}(\T^d;\RRd)) \quad \text{a.s.},
\end{equation}
for $\theta \in [0,\half)$,  $r,\zeta\in (2,\infty)$. Thus by Theorem \ref{thm:blow-up_crit_critical} applied to \eqref{eq:v}, we have
\begin{equation}\label{eq:para_indep4}
    \mathbf{P}\big(\tau < \infty,\sup\limits_{t \in [t_0, \tau)} \|v(t)\|_{{\BB^{\beta_0}_{q_0,p_0}(\T^d;\RR^d)}}+\|v\|_{L^{p_0}(t_0, \tau; \Hs^{\tilde{s}_0,q_0}(\T^d;\RR^d))}<\infty\big) = 0.
\end{equation}
Let $\mathcal{V}:=\{\sigma>t_0\}$. Since $\tau>t_0$ almost surely, \eqref{eq:para_indep4} implies \ref{it:2} in the statement of the theorem if
\begin{equation}\label{eq:claimblowup}
    \tau =\sigma \text{ a.s. on }\mathcal{V}\quad \text{ and }\quad u=v \text{ a.e. on }[t_0,\sigma)\times \mathcal{V}.
\end{equation}
We now show that \eqref{eq:claimblowup} holds. First, note that $(u|_{[t_0, \sigma)\times \mathcal{V}}, \ind_{\mathcal{V}}\sigma+\ind_{\Omega\setminus\mathcal{V}}t_0)$ is a $(p_0, q_0,\kappa_0,s_0)$-solution to \eqref{eq:v}. Since $(v,\tau)$ is a maximal solution (see \cite[Definition 4.6(3)]{AV25_survey}), we have
\begin{equation}\label{eq:max_v}
    \sigma\leq \tau \text{ a.s. on }\mathcal{V}\quad \text{ and }\quad  u=v \text{ a.e. on }[t_0,\sigma)\times \mathcal{V}.
\end{equation}
Therefore, it remains to show that $\mathbf{P}(\mathcal{V}\cap \{\sigma< \tau\})=0$. To this end, note that \eqref{eq:instan_v} and \eqref{eq:max_v} yield $u=v \in L^p_\loc((t_0, \sigma]; \Hs^{\tilde{s},q}(\T^d; \RR^d))$ a.s. on $\mathcal{V}\cap \{\sigma<\tau\}$. In addition, combining this with \eqref{eq:u_instan} shows that $u\in L^p(0, \sigma; \Hs^{\tilde{s},q}(\T^d;\RRd))$ a.s. on $\mathcal{V}\cap \{\sigma<\tau\}$. Similarly, we obtain that $\sup_{t\in [0, \sigma)}\|u(t)\|_{\BB^{\frac{d}{q}+1-2\gam}_{q,p}(\T^d;\RRd)}<\infty$ a.s. on $\mathcal{V}\cap \{\sigma<\tau\}$. Therefore, using \eqref{eq:para_indep1} we obtain
\begin{align*}
    \mathbf{P}&(\mathcal{V}\cap \{\sigma<\tau\})\\&= \mathbf{P}\Big(\mathcal{V}\cap \{\sigma<\tau\}\cap \big\{\sup_{t\in [0, \sigma)}\|u(t)\|_{\BB^{\frac{d}{q}+1-2\gam}_{q,p}(\T^d;\RRd)}+ \|u\|_{L^p(0, \sigma; \Hs^{\tilde{s},q}(\T^d;\RRd))}<\infty\big\}\Big)\\
    &\leq \mathbf{P}\Big(\sigma<\infty, \sup_{t\in [0, \sigma)}\|u(t)\|_{\BB^{\frac{d}{q}+1-2\gam}_{q,p}(\T^d;\RRd)}+ \|u\|_{L^p(0, \sigma; \Hs^{\tilde{s},q}(\T^d;\RRd))}<\infty\Big)=0.
\end{align*}
This completes the proof of \ref{it:2}.

\textit{Proof of \ref{it:1}}. Again, let $p_0\in (2,\infty)$, $q_0\in (2,\infty)$ and $s_0\in [\gam, 2\gam)$ be such that the assumptions of Theorem \ref{thm:localWP} hold with these parameters. Let $q_1>q_0$ and take  $\tilde{q}\in (q_0, q_1)$ be such that the assumptions of Theorem \ref{thm:localWP} also hold with $q_0$ replaced by $\tilde{q}$. Let $\kappa_0= \kappa_{c,0}$ as in \eqref{eq:kappa0} and define 
\begin{equation*}
    \kappa_{c,1}:= -1  + \frac{p_0}{2}\big(4-\gam^{-1}\big(\frac{d}{\tilde{q}}+s_0+1\big)\big).
\end{equation*}
In addition, take $\kappa\in (\kappa_{c,0}, \min\{\kappa_{c,1}, \frac{p_0}{2}-1\})$ and let $\beta_1:=-s_0+2\gam(1-\frac{\kappa+1}{p_0})$ be the regularity of the trace space. Then, since $\kappa> \kappa_{c,0}$, we have $\beta_1<\beta_0<\gam$ (see also \eqref{eq:beta0}), so that $\BB^{\beta_1}_{\tilde{q}, p_0}(\T^d;\RR^d)$ is non-critical. Arguing as in the proof of \ref{it:2} and using instantaneous regularisation, take $\lambda\in (\beta_0,\gam)$ and we obtain 
\begin{equation*}
    \ind_{\{\sigma>t_0\}}u(t_0)\in L^0_{\mathcal{F}_{t_0}}(\Omega; \BB^{\beta_1}_{\tilde{q}, p_0}(\T^d; \RR^d)).
\end{equation*}
Hence, by the local well-posedness from Theorem \ref{thm:localWP}, there exists a $(p_0,\tilde{q},\kappa,s_0)$-maximal solution $(v, \tau)$ on $[t_0,\infty)$ to \eqref{eq:v}. Since the setting is non-critical, Theorem \ref{thm:blow-up_crit_noncritical} yields 
\begin{equation*}
    \mathbf{P}\big(\tau < \infty,\sup\limits_{t \in [t_0, \tau)} \|v(t)\|_{{\BB^{\beta_1}_{\tilde{q},p_0}(\T^d;\RR^d)}}<\infty\big) = 0.
\end{equation*}
Since $\beta_1<\beta_0$, we have the embeddings $B^{\beta_0}_{q_1,\infty}\hookrightarrow B^{\beta_0}_{\tilde{q},\infty}\hookrightarrow B^{\beta_1}_{\tilde{q}, p_0}$ and therefore
\begin{equation*}
    \mathbf{P}\big(\tau < \infty,\sup\limits_{t \in [t_0, \tau)} \|v(t)\|_{{\BB^{\beta_0}_{q_1,\infty}(\T^d;\RR^d)}}<\infty\big) = 0.
\end{equation*}
It remains to prove that $\tau=\sigma$ a.s. on $\mathcal{V}:=\{\sigma>t_0\}$ and $u=v$ a.e. on $[t_0, \sigma)\times \mathcal{V}$, see \eqref{eq:claimblowup}. This is proved as in \ref{it:2}.
\end{proof}

\subsection{The vorticity equation}\label{subsec:vorticity}
In this section, we study well-posedness of the vorticity equation
\begin{equation}\label{eq: secondvorticity}
       \left\{
        \begin{aligned}
            &\dd \xi + (-\del)^\gam \xi \dd t  = -(\BS\xi \cdot \grad)\xi \dd t + \sum_{n \geq 1} \left(\alpha_n\xi + (\nabla \alpha_n) \times \BS\xi\right)\dd W^n_t&\text{ for }&t\geq 0,\\
            &\xi|_{t=0}=\xi_0,
        \end{aligned}
        \right.
\end{equation}
as motivated in Section \ref{subsec:roadmap}. The sequence $(\alpha_n)_{n \geq 1}$ is chosen to be any square-summable sequence in $W^{1,\infty}(\T^2;\R)$. Given that the curl of a vector field is always zero-mean, we will henceforth frequently refer to the zero-mean subspaces and denote them with a dot, for example 
\begin{equation*}
    \dot{H}^{s,q}(\T^2;\R)\quad\text{ and }\quad \dot{\Hs}^{s,q}(\T^2;\R^2).
\end{equation*}
In addition, we recall the Biot--Savart operator $\BS$ which acts as inverse to the curl operator, see \cite{bertozzi2002vorticity} for an elaborate introduction. To be precise, $\BS f$ is defined by the unique divergence-free and zero-mean vector field $g:\T^2 \rightarrow \R^2$ such that
$\nabla \times g = f$, or, more explicitly, $\BS f = \nabla^\perp \Delta^{-1}f$. 
Moreover,
\begin{equation*}
    \BS: \dot{H}^{s,p}(\T^2;\R)\to \dot{\Hs}^{s+1,p}(\T^2;\R^2)
\end{equation*}
is a bounded linear operator for any $s \in \R$ and $p\in (1,\infty)$.
 We note that as the operator $\BS$ is regularising, one can simply repeat all of the prior calculations in this paper for the equation \eqref{eq: secondvorticity}, replacing the divergence-free spaces by zero-mean spaces. In particular, we have the following local well-posedness for the vorticity equation.

\begin{theorem}[Local well-posedness of the stochastic hypoviscous Navier--Stokes equations in vorticity form]\label{thm:localWPvorticity}
    Let $\gam\in (\frac 12,1]$, $s\in [\gam, 2\gam)$. Furthermore, assume that $p,q\in (2,\infty)$ satisfy
      \begin{equation*}
        \max\{|s-1|, 3\gam-s-1\}<\frac 2q < 4\gam-s-1\quad \text{ and }\quad \frac{2}{p}+\frac{2}{q\gam}\leq 4-\frac{s+1}{\gam}.
    \end{equation*}
    Moreover, set $\kappa:=-1 + \frac{p}{2}\big(4-\gam^{-1}(\frac{2}{q}+s+1)\big)$. Then for all $\xi_0 \in L^0_{\mc{F}_0}(\Om; \dot{B}_{q,p}^{\frac{2}{q}+1-2\gam}(\T^2;\R))$ there exists a unique maximal solution $(\xi,\Theta)$ to \eqref{eq: secondvorticity} satisfying $\Theta>0$ a.s. and
    \begin{align*}
        \xi&\in H^{\theta,p}_{\loc}([0,\Theta),w_{\kappa}; \dot{H}^{-s+2\gam(1-\theta),q}(\T^2; \R)) \,\,\text{a.s. for all }\theta\in [0,1/2),\\
        \xi&\in C([0, \Theta); \dot{B}_{q,p}^{\frac{2}{q}+1-2\gam}(\T^2; \R))\,\, \text{ a.s.}
    \end{align*}
    The solution $(\xi,\Theta)$ instantaneously regularises in time and space:
    \begin{equation*}
    \xi\in L^r_\loc((0,\Theta); \dot{H}^{\gam,\zeta}(\T^2;\R))\cap C^{\theta -\eps}_\loc((0,\Theta); \dot{H}^{-\gam+2\gam(1-\theta),\zeta}(\T^2;\R)) \quad \text{a.s.},
\end{equation*}
for $\theta \in (0,\half)$, $\eps\in (0,\theta)$,  $r,\zeta\in (2,\infty)$.
\end{theorem}

The real work to be done in this section is establishing global well-posedness of the vorticity equation. 

\begin{theorem}[Global well-posedness of the stochastic hypoviscous Navier--Stokes equations in vorticity form]\label{thm:globalWPvorticity}
In addition to the conditions of Theorem \ref{thm:localWPvorticity}, assume that $2/q< 2\gam-1$. Then the maximal solution $(\xi,\Theta)$ of \eqref{eq: secondvorticity} is global.
\end{theorem}

\begin{proof}
Note that we are in the regime of parameters motivated in Section \ref{subsec:roadmap}, whereby $L^\zeta \subseteq B^{2/q+1-2\gam}_{q_1,\infty}$ for any $q < q_1 \leq \zeta$. Thus, the central idea of the proof is the $L^\zeta$-estimate of $\xi$. To obtain such an estimate, we wish to apply the generalised It\^{o} formula from \cite[Proposition A.1]{debussche2016degenerate}, though the process $\xi$ lacks sufficient regularity at time zero. To accommodate this, we will instead begin our solution from an arbitrary positive time $t>0$; this will require the use of localising times as, \textit{a priori}, one could have that $0 < \Theta < t$. To this end, let us fix a deterministic time horizon $T>0$, let $(\theta_j)_{j\geq 1}$ be a localising sequence for the maximal solution $(\xi, \Theta)$, and introduce the first hitting times
    $$\tau_j \coloneqq \theta_j \wedge T \wedge \inf\big\{t' \geq t: \sup_{r\in [t,t']}\norm{\xi_{r}}_{L^{\infty}} + \int_t^{t'}\norm{\xi_{r}}_{H^{\gamma, \eta}}^2\dd r \geq j \big\},$$
    where $\eta$ is a henceforth fixed parameter such that $\eta > \frac{4}{\gamma}$ and $\eta \geq \zeta$. Then $\xi$ is a solution on $[0,\tau_j]$ and from the regularity given in Theorem \ref{thm:localWPvorticity}, the sequence $(\tau_j)_{j\geq 1}$ converges to $\Theta \wedge T$ a.s. (observe that $\eta$ is chosen sufficiently large so that $H^{\frac{\gamma}{2}, \eta}$ embeds into $L^\infty$). Note the usual convention that the infimum of the empty set is infinite. In particular, for every $t < t' \leq T$ and all $j$, we have that
\begin{align*}
    \xi_{t' \wedge \tau_j}  &= \xi_{t \wedge \tau_j} - \int_{t \wedge \tau_j}^{t'\wedge\tau_j}(-\del)^\gam \xi_r\dd r -\int_{t \wedge \tau_j}^{t'\wedge\tau_j}(\BS\xi_r \cdot \grad)\xi_r \dd r\\ &+ \sum_{n \geq 1} \int_{t \wedge \tau_j}^{t'\wedge\tau_j}\left(\alpha_n\xi_r + (\nabla \alpha_n) \times \BS\xi_r\right)\dd W^n_r.
\end{align*}
We would like to use that the first hitting time $\tau_j$ enforces a pathwise estimate on the $\xi$, however the value of $\norm{\xi_t}_{L^{\infty}}$ could in theory be very large (and in particular, larger than $j$) so to guarantee such a bound we shall also introduce the set
$$\Gamma_{k,j} \coloneqq \left\{\norm{\xi_t}_{L^\infty} \leq k \right\} \cap \left\{ t < \tau_j\right\},$$
where $k$ is another constant. We note that this set belongs to $\mathcal{F}_t$ due to the progressive measurability of $\xi$ and the fact that $\tau_j$ is a stopping times. Multiplying by $\ind_{\Gamma_{k,j}}$, using the $\mathcal{F}_t$ measurability to take the indicator function inside the stochastic integral, and then invoking the (bi)linearity of all operators involved to pass the indicator function onto $\xi$, we have the equality
\begin{align*}
    \hat{\xi}_{t' \wedge \tau_j}  &= \hat{\xi}_{t \wedge \tau_j} - \int_{t \wedge \tau_j}^{t'\wedge\tau_j}(-\del)^\gam \hat{\xi}_r\dd r -\int_{t \wedge \tau_j}^{t'\wedge\tau_j}(\BS\hat{\xi}_r \cdot \grad)\hat{\xi}_r \dd r\\ &+ \sum_{n \geq 1} \int_{t \wedge \tau_j}^{t'\wedge\tau_j}\big(\alpha_n\hat{\xi}_r + (\nabla \alpha_n) \times \BS\hat{\xi}_r\big)\dd W^n_r,
\end{align*}
where $\hat{\xi} = \xi\ind_{\Gamma_{k,j}}$. Of course, on the set $\Gamma_{k,j}$ we have that $t \wedge \tau_j = t$. We apply the It\^{o} formula from \cite[Proposition A.1]{debussche2016degenerate} and proceed similarly to \cite[Proposition 5.1]{debussche2016degenerate}. Using the regularity enforced by the stopping time $\tau_j$ and $\Gamma_{k,j}$, and in particular the fact that $\sup_{t \leq r \leq \tau_j}\|\hat{\xi}_r\|_{L^{\infty}} \leq j \vee k$, then we obtain the identity
\begin{align*}
    \|\hat{\xi}_{t' \wedge \tau_j}\|_{L^{\zeta}}^{\zeta}  &= \|\hat{\xi}_{t \wedge \tau_j}\|_{L^{\zeta}}^{\zeta} - \zeta\int_{t \wedge \tau_j}^{t'\wedge\tau_j}\inner{(-\del)^\gam \hat{\xi}_r}{\hat{\xi}_r|\hat{\xi}_r|^{\zeta-2}}\dd r\\ &-\zeta\int_{t \wedge \tau_j}^{t'\wedge\tau_j}\inner{(\BS\hat{\xi}_r \cdot \grad)\hat{\xi}_r}{\hat{\xi}_r\abs{\hat{\xi}_r}^{\zeta-2}} \dd r\\ &+ \frac{\zeta(\zeta-1)}{2} \sum_{n \geq 1} \int_{t \wedge \tau_j}^{t'\wedge\tau_j}\inner{\abs{\alpha_n\hat{\xi}_r + (\nabla \alpha_n) \times \BS\hat{\xi}_r}^2}{\abs{\hat{\xi}_r}^{\zeta-2}}\dd r\\ &+
    \zeta \sum_{n \geq 1} \int_{t \wedge \tau_j}^{t'\wedge\tau_j}\inner{\alpha_n\hat{\xi}_r + (\nabla \alpha_n) \times \BS\hat{\xi}_r}{\hat{\xi}_r\abs{\hat{\xi}_r}^{\zeta-2}}\dd W^n_r,
\end{align*}
where $\langle\cdot, \cdot\rangle$ denotes the distributional pairing specified by the $L^2$ inner product. For more details we refer the reader to \cite[Lemma 8.21]{AV25_survey} and references therein. We now begin to estimate the drift terms, which can be done by taking smooth approximations of $\hat{\xi}_r$. For a smooth $\tilde{\xi}$, it is classical that
$$\inner{(-\del)^\gam \tilde{\xi}}{\tilde{\xi}\abs{\tilde{\xi}}^{\zeta-2}} \geq 0 ,$$
see \cite[Proposition II.3.23 \& Example II.3.26]{engel2000one} and \cite[Lemma 2.1]{Pr17}. Moreover, in the nonlinear term, we have that
\begin{align*}
    \inner{(\BS\tilde{\xi} \cdot \grad)\tilde{\xi}}{\tilde{\xi}\abs{\tilde{\xi}}^{\zeta-2}} = \frac{1}{\zeta}\int_{\T^2}(\BS\tilde{\xi} \cdot \grad)\abs{\tilde{\xi}}^{\zeta} \dd x = -\frac{1}{\zeta}\int_{\T^2}\textnormal{div}(\BS\tilde{\xi})\abs{\tilde{\xi}}^{\zeta} \dd x = 0
\end{align*}
using that $\textnormal{div}(\BS\tilde{\xi}) = 0$. Thus, by taking such smooth approximations as in \cite[Lemma 8.21]{AV25_survey}, we deduce the inequality
\begin{align*}
    \norm{\hat{\xi}_{t' \wedge \tau_j}}_{L^{\zeta}}^{\zeta}  &\leq \norm{\hat{\xi}_{t \wedge \tau_j}}_{L^{\zeta}}^{\zeta} + \frac{\zeta(\zeta-1)}{2} \sum_{n \geq 1} \int_{t \wedge \tau_j}^{t'\wedge\tau_j}\inner{\abs{\alpha_n\hat{\xi}_r + (\nabla \alpha_n) \times \BS\hat{\xi}_r}^2}{\abs{\hat{\xi}_r}^{\zeta-2}}\dd r\\ &+
    \zeta \sum_{n \geq 1} \int_{t \wedge \tau_j}^{t'\wedge\tau_j}\inner{\alpha_n\hat{\xi}_r + (\nabla \alpha_n) \times \BS\hat{\xi}_r}{\hat{\xi}_r\abs{\hat{\xi}_r}^{\zeta-2}}\dd W^n_r.
\end{align*}
For the quadratic variation term, we have a straightforward bound
\begin{align*}
   &\inner{\abs{\alpha_n\hat{\xi}_r + (\nabla \alpha_n) \times \BS\hat{\xi}_r}^2}{\abs{\hat{\xi}_r}^{\zeta-2}}\\ & \qquad \lesssim \norm{\alpha_n}_{W^{1,\infty}}^2 \inner{\abs{\hat{\xi}_r}^2 + \abs{\BS\hat{\xi}_r}^2}{\abs{\hat{\xi}_r}^{\zeta-2}}\\
   & \qquad \lesssim \norm{\alpha_n}_{W^{1,\infty}}^2\Big(\norm{\hat{\xi}_r}_{L^{\zeta}}^{\zeta} + \Big[\int_{\T^2}\abs{\BS\hat{\xi}_r}^{2 \cdot \frac{\zeta}{2}}\dd x \Big]^{\frac{2}{\zeta}}\Big[\int_{\T^2}\abs{\hat{\xi}_r}^{(\zeta-2) \cdot \frac{\zeta}{\zeta - 2}}\dd x  \Big]^{\frac{\zeta-2}{\zeta}} \Big)\\
   & \qquad \lesssim \norm{\alpha_n}_{W^{1,\infty}}^2\left(\norm{\hat{\xi}_r}_{L^{\zeta}}^{\zeta} + \norm{\BS\hat{\xi}_r}_{L^{\zeta}}^{2}\norm{\hat{\xi}_r}_{L^{\zeta}}^{\zeta-2} \right)\\
   & \qquad \lesssim \norm{\alpha_n}_{W^{1,\infty}}^2\norm{\hat{\xi}_r}_{L^{\zeta}}^{\zeta},
\end{align*}
where we have applied H\"{o}lder's inequality. All in all, we have achieved a bound
\begin{align*}
    \norm{\hat{\xi}_{t' \wedge \tau_j}}_{L^{\zeta}}^{\zeta}  &\lesssim \norm{\hat{\xi}_{t \wedge \tau_j}}_{L^{\zeta}}^{\zeta} +  \int_{t \wedge \tau_j}^{t'\wedge\tau_j}\norm{\hat{\xi}_r}_{L^{\zeta}}^{\zeta}\dd r\\ &+
    \zeta \sum_{n \geq 1} \int_{t \wedge \tau_j}^{t'\wedge\tau_j}\inner{\alpha_n\hat{\xi}_r + (\nabla \alpha_n) \times \BS\hat{\xi}_r}{\hat{\xi}_r\abs{\hat{\xi}_r}^{\zeta-2}}\dd W^n_r
\end{align*}
using that $\sum_{n \geq 1}\norm{\alpha_n}_{W^{1,\infty}}^2$ is just a constant. Following this, we will now take the absolute value of the stochastic integral, the supremum in time and then the expectation, immediately applying the Burkholder--Davis--Gundy inequality as well, to achieve that
\begin{equation}\label{chosen approp}
    \begin{aligned}
  \mathbf{E}\Big[\sup_{t' \in [t,s']} \norm{\hat{\xi}_{t' \wedge \tau_j}}_{L^{\zeta}}^{\zeta}\Big]  &\lesssim \mathbf{E}\Big[\norm{\hat{\xi}_{t \wedge \tau_j}}_{L^{\zeta}}^{\zeta}\Big] +  \mathbf{E}\Big[\int_{t \wedge \tau_j}^{s'\wedge\tau_j}\norm{\hat{\xi}_r}_{L^{\zeta}}^{\zeta}\dd r\Big]\\ &+
      \mathbf{E}\Big[\int_{t \wedge \tau_j}^{s'\wedge\tau_j}\sum_{n \geq 1}\inner{\alpha_n\hat{\xi}_r + (\nabla \alpha_n) \times \BS\hat{\xi}_r}{\hat{\xi}_r\abs{\hat{\xi}_r}^{\zeta-2}}^2\dd r\Big]^{\frac{1}{2}}. 
\end{aligned}
\end{equation}
By the same computation as for the quadratic variation term, we have that
$$\sum_{n \geq 1}\inner{\alpha_n\hat{\xi}_r + (\nabla \alpha_n) \times \BS\hat{\xi}_r}{\hat{\xi}_r\abs{\hat{\xi}_r}^{\zeta-2}}^2 \lesssim \norm{\hat{\xi}_r}_{L^{\zeta}}^{2\zeta}$$
and furthermore
\begin{align*}
         c\mathbf{E}\Big[\int_{t \wedge \tau_j}^{s'\wedge\tau_j}\norm{\hat{\xi}_r}_{L^{\zeta}}^{2\zeta}\dd r\Big]^{\frac{1}{2}} &\leq c\mathbf{E}\Big[\sup_{t'\in[t\wedge\tau_j,s'\wedge\tau_j]}\norm{\hat{\xi}_{t'}}_{L^{\zeta}}^{\zeta}\int_{t \wedge \tau_j}^{s'\wedge\tau_j}\norm{\hat{\xi}_r}_{L^{\zeta}}^{\zeta}\dd r\Big]^{\frac{1}{2}}\\
         &= c\mathbf{E}\Big[\sup_{t'\in[t,s']}\norm{\hat{\xi}_{t'\wedge\tau_j}}_{L^{\zeta}}^{\zeta}\int_{t \wedge \tau_j}^{s'\wedge\tau_j}\norm{\hat{\xi}_r}_{L^{\zeta}}^{\zeta}\dd r\Big]^{\frac{1}{2}}\\
         &\leq c\mathbf{E}\Big[\frac{1}{4c^2}\sup_{t'\in[t,s']}\norm{\hat{\xi}_{t'\wedge\tau_j}}_{L^{\zeta}}^{2\zeta} + \frac{4c^2}{3}\Big(\int_{t \wedge \tau_j}^{s'\wedge\tau_j}\norm{\hat{\xi}_r}_{L^{\zeta}}^{\zeta}\dd r\Big)^2\Big]^{\frac{1}{2}}\\
         &\leq c\mathbf{E}\Big[\frac{1}{2c}\sup_{t'\in[t,s']}\norm{\hat{\xi}_{t'\wedge\tau_j}}_{L^{\zeta}}^{\zeta} + \frac{2c}{\sqrt{3}}\Big(\int_{t \wedge \tau_j}^{s'\wedge\tau_j}\norm{\hat{\xi}_r}_{L^{\zeta}}^{\zeta}\dd r\Big)\Big]\\
         &= \frac{1}{2}\mathbf{E}\Big[\sup_{t'\in[t,s']}\norm{\hat{\xi}_{t'\wedge\tau_j}}_{L^{\zeta}}^{\zeta}\Big] + \frac{2c^2}{\sqrt{3}}\mathbf{E}\Big[\int_{t \wedge \tau_j}^{s'\wedge\tau_j}\norm{\hat{\xi}_r}_{L^{\zeta}}^{\zeta}\dd r\Big],
\end{align*}
where $c$ is any constant, chosen appropriately for the inequality \eqref{chosen approp}. Returning to this bound, taking the one half term to the other side, we have simply that
    \begin{align}
 \nonumber   \mathbf{E}\Big[\sup_{t' \in [t,s']} \norm{\hat{\xi}_{t' \wedge \tau_j}}_{L^{\zeta}}^{\zeta}\Big]  \lesssim \mathbf{E}\left[\norm{\hat{\xi}_{t \wedge \tau_j}}_{L^{\zeta}}^{\zeta}\right] +  \mathbf{E}\Big[\int_{t \wedge \tau_j}^{s'\wedge\tau_j}\norm{\hat{\xi}_r}_{L^{\zeta}}^{\zeta}\dd r\Big],
\end{align}
which of course implies the coarser bound
    \begin{align*}
   \mathbf{E}\Big[\sup_{t' \in [t,s']} \norm{\hat{\xi}_{t' \wedge \tau_j}}_{L^{\zeta}}^{\zeta}\Big]  
 &\lesssim \mathbf{E}\Big[\norm{\hat{\xi}_{t \wedge \tau_j}}_{L^{\zeta}}^{\zeta}\Big] +  \mathbf{E}\Big[\int_{t}^{s'}\norm{\hat{\xi}_{r \wedge \tau_j}}_{L^{\zeta}}^{\zeta}\dd r\Big]\\
&\lesssim \mathbf{E}\Big[\norm{\hat{\xi}_{t \wedge \tau_j}}_{L^{\zeta}}^{\zeta}\Big] +  \int_{t}^{s'}\mathbf{E}\Big[\norm{\hat{\xi}_{r \wedge \tau_j}}_{L^{\zeta}}^{\zeta}\Big]\dd r
\\
&\lesssim \mathbf{E}\Big[\norm{\hat{\xi}_{t \wedge \tau_j}}_{L^{\zeta}}^{\zeta}\Big] +  \int_{t}^{s'}\mathbf{E}\Big[\sup_{t' \in [t,r]}\norm{\hat{\xi}_{t' \wedge \tau_j}}_{L^{\zeta}}^{\zeta}\Big]\dd r.
\end{align*}
Therefore, we may apply the classical Gr\"{o}nwall inequality to deduce that
$$\mathbf{E}\Big[\sup_{t' \in [t,T]} \norm{\hat{\xi}_{t' \wedge \tau_j}}_{L^{\zeta}}^{\zeta}\Big]  
 \lesssim \mathbf{E}\Big[\norm{\hat{\xi}_{t \wedge \tau_j}}_{L^{\zeta}}^{\zeta}\Big],$$
where the constant depends on $T$ but is independent of $j$. By applying the monotone convergence theorem, using that $\tau_j$ is monotonically increasing to $\Theta \wedge T$, we deduce that
$$\mathbf{E}\Big[\ind_{\Gamma_k}\sup_{t' \in [t,\Theta \wedge T)} \norm{\xi_{t'}}_{L^{\zeta}}^{\zeta}\Big]  
 \lesssim \mathbf{E}\Big[\ind_{\Gamma_k}\norm{\xi_{t}}_{L^{\zeta}}^{\zeta}\Big],$$
 where $\Gamma_k \coloneqq \{t < \Theta\} \cap \left\{\norm{\xi_t}_{L^\infty} \leq k\right\}$, having used on the right hand side that $t < \Theta$ on the set $\Gamma_k$ and noting that $t < T$. Of course, for every fixed $k$, by definition of $\Gamma_k$, the right hand side is finite. Thus, we have that $\sup_{t' \in [t,\Theta \wedge T)} \norm{\hat{\xi}_{t'}}_{L^{\zeta}} $ is finite a.s. on $\Gamma_k$ for every $k$. However, we also note that from the regularity of $\xi$, the sets $\Gamma_k$ are increasing and converge to the set $\left\{t < \Theta\right\}$. 
Therefore, a.s. on the whole set $\left\{t < \Theta\right\}$, we have that $\sup_{t' \in [t,\Theta \wedge T)} \norm{\xi_{t'}}_{L^{\zeta}} < \infty$. In particular,
$$\mathbf{P}\big(t < \Theta, \sup_{t' \in [t,\Theta \wedge T)} \norm{\xi_{t'}}_{L^{\zeta}} < \infty  \big) = \mathbf{P}\left(t < \Theta\right).$$
However, by the motivating blow-up criterion of Theorem \ref{thm:indep_parameter_blow-up} part \ref{it:1}, combined with the embedding $L^\zeta \subseteq B^{2/q+1-2\gam}_{\zeta,\infty}$, then
$$\mathbf{P}\big(t < \Theta < T, \sup_{t' \in [t,\Theta)} \norm{\xi_{t'}}_{L^{\zeta}} < \infty  \big) = 0.$$
For notational convenience, let us denote the set
$$A_t \coloneqq \big\{t < \Theta, \sup_{t' \in [t,\Theta \wedge T)} \norm{\xi_{t'}}_{L^{\zeta}} < \infty \big\}.$$
Note that
$$\big\{t < \Theta < T, \sup_{t' \in [t,\Theta)} \norm{\xi_{t'}}_{L^{\zeta}} < \infty  \big\} = A_t \cap \left\{ \Theta < T \right\}$$
and
$$\lim_{t \rightarrow 0} \mathbf{P}(A_t) = \lim_{t \rightarrow 0}\mathbf{P}\left(t < \Theta \right) = \mathbf{P}(0 < \Theta) = 1,$$
using monotonicity of the sets $\left\{t < \Theta \right\}$ in $t$. Combining these properties, observe that
\begin{align} \nonumber
    \mathbf{P}\left(\Theta < T\right) &= \mathbf{P}\left(\left\{\Theta < T \right\} \cap A_t \right) + \mathbf{P}\left(\left\{\Theta < T \right\} \cap A_t^{\rm c} \right)\\ &= \mathbf{P}\left(\left\{\Theta < T \right\} \cap A_t^{\rm c} \right)
    \leq \mathbf{P}\left( A_t^{\rm c} \right)
    \longrightarrow 0 \label{identically}
\end{align}
in the limit $t \rightarrow 0$. Therefore, for every $T>0$ we have that $\mathbf{P}\left(\Theta < T \right) = 0$. Hence, $\Theta = \infty$ a.s. and the solution is global.
\end{proof}

\subsection{Global well-posedness}\label{subsec:proofglobal}
Finally, we will transfer the solution theory for the vorticity equation \eqref{eq: secondvorticity} to our original hypoviscous Navier--Stokes equation \eqref{eq:abstract_eq}. As motivated in Section \ref{subsec:roadmap}, we are interested in a noise $g_n$ of the form
$$g_n(x,u) = \alpha_n(x)u.$$
However, as we wish to identify this term through applying the Biot--Savart operator, we must also impose a zero-mean constraint on it. Therefore, we introduce the projection $\Pi$ onto the zero-mean subspace, defined for $f=(f^n)_{n=1}^2$ by
$$(\Pi f)^n = f^n - \int_{\T^2}f^n(x)\dd x.$$
This can equivalently be defined with respect to the $k$-th Fourier coefficient $\widehat{f^n}(k)$ by
\begin{equation*}
\widehat{(\Pi f)^n}(k)
:=
\widehat{f^n}(k),
\qquad
k\in\mathbb{Z}^2\setminus\{0\},
\qquad
\widehat{(\Pi f)^n}(0)
:=
0.
\end{equation*}
Thus, we actually treat the equation \eqref{eq:abstract_eq} in the specific case
\begin{equation} \label{global noise}g_n(x,u):= \Pi[\alpha_n(x)u]\quad \text{ with }(\alpha_n)_{n\geq 1}\in \ell^2(W^{1,\infty}(\T^2)).\end{equation}
It should be noted that the curl operation passes through the projection $\Pi$, i.e.,
$$\nabla \times (\Pi f) = \nabla \times f$$
as
\begin{align*}
    \nabla \times f &= \nabla \times \Big(f - \int_{\T^2}f(x)\dd x \Big) + \nabla \times  \Big(\int_{\T^2}f(x)\dd x\Big)\\& = \nabla \times \Big(f - \int_{\T^2}f(x)\dd x \Big) = \nabla \times \left(\Pi f \right).
\end{align*}
This is of course very similar to the justification that the curl passes through the Helmholtz projection, using the decomposition $f = \P f + \nabla g$ and the fact that $\nabla \times (\nabla g) = 0$. The computation \eqref{explicitly compute} is therefore still valid with our choice of $g_n$ including the projector $\Pi$.\\

We state the following lemma connecting the solution to the vorticity equation \eqref{eq: secondvorticity} and the velocity equation \eqref{eq:abstract_eq}.

\begin{lemma} \label{lemma:vorticity_to_velocity}
    Let $(\xi,\Theta)$ be the unique maximal solution to \eqref{eq: secondvorticity}, specified in Theorem \ref{thm:localWPvorticity}. Then $(\BS\xi, \Theta)$ is a local solution to \eqref{eq:abstract_eq} for the choice of $g_n$ given by \eqref{global noise}. Furthermore let $(u,\sigma)$ denote the unique maximal solution to \eqref{eq:abstract_eq} with the initial condition $u_0 \coloneqq \BS\xi_0$. Then $\Theta \leq \sigma$ a.s. and $u = \BS\xi$ on $[0,\Theta)$. 
\end{lemma}

\begin{proof}
    As $\BS\xi$ has only better regularity than $\xi$, we just need to verify that the identity is satisfied. We recall that for a zero-mean function $f$, $\BS f$ is the unique divergence-free and zero-mean vector field such that $\nabla \times (\BS f) = f$. Using that the equation \eqref{eq: secondvorticity} was obtained by taking the curl of $u$, we have already seen that
    \begin{align*}
        \nabla \times \left[(-\del)^\gam \BS\xi \right] &= (-\del)^\gam \xi,\\
        \nabla \times \left[\P \left((\BS\xi \cdot \grad)\BS\xi \right)\right] &= (\BS\xi \cdot \grad)\xi,\\
        \nabla \times \left[ \P \Pi \left(\alpha_n \BS\xi\right)\right] &= \alpha_n\xi + (\nabla \alpha_n) \times \BS\xi.
    \end{align*}
   Thus applying $\BS$ term by term to the identity satisfied by $\xi$, we see that on any localising sequence for $\Theta$, the identity
   $$\dd \BS\xi + [(-\del)^\gam \BS\xi+\P\left((\BS\xi\cdot\grad) \BS\xi\right)]\dd t  =\sum_{n\geq 1}\P \Pi(\alpha_n\BS\xi)\dd W^n_t$$
holds. Hence, $(\BS\xi, \Theta)$ is a local solution to \eqref{eq:abstract_eq}. The second part of the lemma now simply follows from the definition of the maximal time and uniqueness; maximality ensures that $\Theta \leq \sigma$ a.s., whilst uniqueness implies that $\BS\xi = u$ on $[0,\Theta)$.
\end{proof}

In fact, from the regularising properties of $\BS$, one can see that the solution $(\BS\xi, \Theta)$ has more regularity than that observed in Theorem \ref{thm:localWP}. Of course this has come at the cost of a smoother initial condition, requiring not only $u_0 \in L^0_{\mc{F}_0}(\Om; \BB_{q,p}^{2/q+1-2\gam}(\T^2;\RR^2))$ but also $\nabla \times u_0 \in L^0_{\mc{F}_0}(\Om; \dot{B}_{q,p}^{2/q+1-2\gam}(\T^2;\R))$. With Lemma \ref{lemma:vorticity_to_velocity} in hand, we see that global well-posedness of the solution $(u,\sigma)$ will follow from global well-posedness of $(\xi,\Theta)$.

\begin{proof}[Proof of Theorem \ref{thm:globalWPvelocity}.]
        Due to Theorem \ref{thm:globalWPvorticity}, there exists a global solution $(\xi,\Theta)$ to \eqref{eq: secondvorticity} with initial condition $\xi_0 \coloneqq \nabla \times u_0.$ Since $u_0 = \BS\xi_0$, Lemma \ref{lemma:vorticity_to_velocity} yields $\Theta \leq \sigma$ a.s., hence $\sigma = \infty$ a.s. so the solution $(u,\sigma)$ is global.
\end{proof}

\bibliographystyle{plain}
\bibliography{references}
\end{document}